\documentclass[11pt]{amsart}
\usepackage{amsmath}
\usepackage{amsfonts}
\usepackage{latexsym}
\usepackage{amssymb, mathabx, amscd}
\usepackage[pdftex]{graphicx}
\usepackage{enumerate}
\usepackage{tikz}
\usepackage{comment}
\usepackage[all,cmtip]{xy}

\usepackage[alphabetic]{amsrefs}

\usepackage[margin=1in]{geometry}
\usepackage{listings}
\usepackage{courier}
\usepackage{tikz-cd}
\usepackage{xcolor}
\usepackage{hyperref}
\usepackage{mathrsfs, colonequals}

\newcommand{\homtriv}[2]{\CH^{#2}(#1)_{{\mathrm{hom}}=0}}

\newcommand{\cc}{\mathbb{C}}

\newcommand{\pp}{\mathbb{P}}

\newcommand{\zz}{\mathbb{Z}}
\newcommand{\ff}{\mathbb{F}}

\newcommand{\OO}{\mathcal{O}}

\newcommand{\F}{\mathcal{F}}

\newcommand{\symgp}{\mathcal{S}}

\newcommand{\FF}{\mathbb{F}} 
 
\newcommand{\Z}{{\mathbf Z}}

\newcommand{\Zhat}{\hat{\mathbb{Z}}}

\newcommand{\ord}{\text{ord}}

\newcommand{\GSp}{\operatorname{GSp}}

\newcommand{\Sp}{\text{Sp}}

\newcommand{\Pic}{\operatorname{Pic}}
\newcommand{\Jac}{\operatorname{Jac}}

\newcommand{\MD}{\operatorname{MD}}

\newcommand{\ra}{\rightarrow}

\newcommand{\sgn}{\operatorname{sgn}}

\newcommand{\Frob}{\operatorname{Frob}}

\newcommand{\Aut}{\operatorname{Aut}}
\newcommand{\Out}{\operatorname{Out}}

\newcommand{\CH}{\operatorname{CH}}
\newcommand{\Sh}{\operatorname{Sh}}

\newcommand{\Cer}{\operatorname{Cer}}

\title{Certifying nontriviality of Ceresa classes of curves}

\author[Ellenberg]{Jordan Ellenberg}
\address{Current: J.~Ellenberg, 325 Van Vleck Hall, Department of Mathematics, University of Wisconsin, 480 Lincoln Drive, Madison, WI 53706, USA}
\email{ellenber@math.wisc.edu}

\author[Logan]{Adam Logan}
\address{Current: A.~Logan, Department of Pure Mathematics, 200 University Avenue West, Waterloo, ON N2L 3G1, Canada}
\email{adam.m.logan@gmail.com}

\author[Srinivasan]{Padmavathi Srinivasan}
\address{Current: P.~Srinivasan, Boston University, 665 Commonwealth Avenue, Boston, MA 02215, USA}
\email{padmask@bu.edu}

\subjclass[2010]{Primary .}

\date{\today}

\usepackage{algorithmic}
\usepackage{amsmath}
\usepackage{amssymb}
\usepackage{amsthm}
\usepackage{amsrefs}
\usepackage{caption}
\usepackage{changepage}
\usepackage{tikz-cd}
\usepackage{fullpage}
\usepackage{colonequals}

\definecolor{blueshade}{RGB}{0, 0, 204}
\definecolor{forestgreen}{RGB}{0, 119, 85}
\definecolor{purple(x11)}{rgb}{0.63, 0.36, 0.94}
\definecolor{chocolate(web)}{rgb}{0.82, 0.41, 0.12}

\newtheorem{thm}{Theorem}[section]

\newtheorem{lem}[thm]{Lemma}

\newtheorem{prop}[thm]{Proposition}
\newtheorem{cor}[thm]{Corollary}

\newtheorem{question}[thm]{Question}
\newtheorem{algo}[thm]{Algorithm}

\newcommand{\defi}[1]{{\color{purple(x11)}{{\textsf{#1}}}}} 

\theoremstyle{definition}
\newtheorem{defin}[thm]{Definition}
\newtheorem{example}[thm]{Example}

\newtheorem{notation}[thm]{Notation}

\theoremstyle{remark}
\newtheorem{rem}{Remark}

\DeclareMathOperator{\lcm}{lcm}
\DeclareMathOperator{\Stab}{Stab}
\DeclareMathOperator{\Griff}{Griff}
\DeclareMathOperator{\chara}{char}

\begin{document}

\maketitle

\date{\today}

\begin{abstract}
The Ceresa cycle is a canonical algebraic $1$-cycle on the Jacobian of an algebraic curve. We construct an algorithm which, given a curve over a number field, often provides a certificate that the Ceresa cycle is non-torsion, without relying on the presence of any additional symmetries of the curve.  Under the hypothesis that the Sato--Tate group is
the whole of $\GSp$, we prove that if the Ceresa class (the image of the Ceresa cycle in \textup{\'{e}t}ale cohomology) is non-torsion, then the algorithm will eventually terminate with a certificate attesting to this fact.  
\end{abstract}

\section{Introduction}

The Ceresa cycle is an algebraic cycle canonically associated to an algebraic curve.  Its name was attached when Ceresa proved in \cite{ceresa} that the Ceresa cycle of a generic genus $g$ curve is not algebraically equivalent to $0$ when $g \geq 3$.  This bound is sharp: it turns out that the Ceresa cycle of a hyperelliptic curve is always trivial, so Ceresa's theorem cannot hold in genus $1$ or $2$.

The Ceresa cycle (along with its close cousin, the Gross--Schoen cycle) has proven to be of arithmetic and geometric importance in quite a few domains.  For instance, when $C$ is a curve over a number field $K$, the torsionness of the cycle is related to the Galois action on the fundamental group~\cite{hainmatsumoto}; and when $C$ is a modular curve, the height of the Ceresa cycle in the Chow group is related to triple product $L$-functions~\cite{YZZ,ZhangGS}. Computations with the Ceresa cycle of the Fermat quartic lead to the formulation of the Beilinson--Bloch conjectures about algebraic cycles, a higher dimensional analogue of the Birch and Swinnerton--Dyer conjecture. Provably nontrivial algebraic cycles that are homologically trivial are hard to write down, and the Ceresa cycle remains one of the naturally occurring sources of nontrivial algebraic cycles that most frequently appear in numerical evidence gathered to support the Beilinson--Bloch conjectures \cite{BST}.    

Explicit computations concerning the Ceresa cycle have been rare, for all its theoretical importance.  The Ceresa cycle is a $1$-cycle in the $3$-fold $C \times C \times C$, and as with any problem in Chow groups of codimension higher than $1$, the relations of algebraic and rational equivalence are not easy to understand concretely.  There are examples of special classes of curves exploiting the presence of symmetries on the curve such as additional correspondences on the Jacobian. Notably,  Fermat curves and some low-degree Fermat quotients have been shown to have non-torsion Ceresa cycle ~\cite{Bloch1, BHarris, EskandariMurty, Kimura, Tadokoro, Otsubo} and other more recent results of Qiu and Zhang give examples of non-hyperelliptic curves whose Ceresa cycle is torsion~\cite{qiuzhang}. In \cite{BST}, Buhler, Schoen and Top gather evidence for the Beilinson--Bloch conjectures by studying the Ceresa cycles of genus $3$ curves embeddded in a triple product of elliptic curves. Recent work of Laga and Shnidman \cite{laga-shnidmanbiellipticpicard} completely settles the case of bielliptic Picard curves (genus $3$ curves of the form $y^3 = x^4 + ax^2 + b$.)  In this family, the Ceresa cycle is always torsion mod algebraic equivalence, and it is torsion mod rational equivalence if and only if an explicitly computable point on a related elliptic curve is torsion. A recent preprint by Kerr, Li, Qiu and Yang~\cite{KLQY} shows that many modular curves have Ceresa cycle which is non-torsion modulo rational equivalence. Recently announced work by Gao and Zhang~\cite{GaoZhangNorthcott} (see also \cite{Hain, KerrTayou}) demonstrates that the height of the Ceresa cycle obeys a Northcott property away from a closed locus $F_g$ in the moduli space of curves of genus $g$. They show that $F_g$ is a proper subset for $g>2$, and hence, for all $d \ge 1$ and $g \ge 3$, the set of curves of genus $g$ whose moduli point is not in $F_g$ and whose Ceresa cycle is torsion up to rational equivalence has only finitely many points defined over the union of all number fields of degree $d$.

The main content of the present work is to construct an algorithm which, given a curve over a number field, often provides a certificate that the Ceresa cycle is non-torsion, without relying on additional symmetries of the curve.  Our algorithm, in the version presented here, cannot provide a proof that the Ceresa cycle is torsion, but only very compelling evidence for it. However, under the hypothesis that the Sato-Tate group is
the whole of $\GSp_{2g}$, we prove in Theorem \ref{T:shadowcheb} that if the Ceresa class (the image of the Ceresa cycle in \textup{\'{e}t}ale cohomology) is non-torsion, then the algorithm will {\em eventually} terminate with a certificate attesting to this fact.  With some work, Theorem~\ref{T:shadowcheb} could likely be made effective, yielding an algorithm which in finite (but possibly infeasible) time will certify torsionness of the Ceresa class.

Our central tool is an idea that goes back to \cite{BST} and which is also key to the relation between Ceresa cycles and modular curves and Heegner points; we can intersect the Ceresa cycle with a correspondence in $C \times C$ and obtain a point on the Jacobian of $C$, whose nonvanishing is an obstruction to the vanishing of the Ceresa cycle.  But this is only useful if $C$ admits an interesting correspondence.  Given a curve $C$ over a number field, our approach is to reduce $C$ modulo many primes $p$; in each such reduction, the graph of Frobenius at $p$ provides a correspondence we can pair with the Ceresa cycle, and the resulting point on $\Jac(C)(\mathbb{F}_p)$ is one we can compute efficiently.  Then vanishing or torsionness of the Ceresa cycle of $C$ will be reflected in the properties of this sequence of points; in particular, we show that if the order of this point in $\Jac(C)(\mathbb{F}_p)$ is large enough, the Ceresa cycle of $C$ must be non-torsion.

As an example of the use of our algorithm in practice, we apply it to a census of
representatives of the $255,564$ signed permutation orbits of smooth plane quartics whose coefficients belong to $\{-1,0,1\}$.  We show that all but $147$ of these have non-torsion Ceresa cycle.  For $142$ of these, we are able to determine by other means whether the Ceresa cycle is torsion or not, leaving $5$ exceptions whose Ceresa cycle we have so far failed to classify.  Of the $9$ curves we show to have torsion Ceresa cycle, all have nontrivial geometric automorphism group. 


We make some scripts in 
Magma \cite{magma} available at \cite{code}
so that the computational statements of the paper can be verified.
The interested reader should consult the \path{README.md} file for detailed instructions on their use.

\section*{Acknowledgments}
This project was started at the Park City Mathematics Institute in 2022 and is based on ideas that were discussed there between the three of us and Akshay Venkatesh.  In particular, insights of his were crucial to the proof of Theorem~\ref{T:shadowcheb}. We would like to thank the organizers of the PCMI program ``Number theory informed by computation" for bringing us together and Henri Darmon, Asvin G, Richard Hain, Jef Laga, Wanlin Li, Bjorn Poonen, Congling Qiu, Ari Shnidman, and Wei Zhang for helpful conversations.  Logan and Srinivasan were supported by Simons Foundation grant 546235 for the collaboration ``Arithmetic Geometry, Number Theory, and Computation''; Srinivasan was also supported by NSF DMS 2401547. 
Logan enjoyed the hospitality of ICERM and the Department of Pure Mathematics at the University of Waterloo for part of the time in which this work was done.  He would also like to thank the Tutte Institute for Mathematics and Computing for its support for and encouragement of his research.  Ellenberg was supported in part by NSF Grant DMS-2301386.

\section*{Notation}
A \defi{nice} variety over a field $k$ is a variety that is smooth, projective, and geometrically integral. We shall assume that $k$ is perfect throughout this paper. The \defi{Chow group of codimension $r$ cycles} (defined over $k$) on a nice variety $Y$ is denoted $\CH^r(Y)$. Following Gross and Schoen \cite[\S~1]{GrossSchoen}, the subgroup of \defi{homologically trivial cycles}  $\homtriv{Y}{r}$ consists of cycles that are  trivial in $\ell$-adic cohomology for all $\ell$. 
The \defi{Griffiths group $\Griff^r(Y)$ of codimension $r$ cycles} defined over $k$  is defined to be the group of
homologically trivial cycles of codimension $r$ modulo those algebraically equivalent to $0$. 
Let $\deg \colon \CH^0(Y) \rightarrow \Z$ denote the \defi{degree} map on $0$-cycles. 
Note that when $Y$ is a nice curve, we have $\CH^1(Y) = \Pic(Y)$ and $\homtriv{Y}{1} = \Pic^0(Y)$,
while $\Griff^1(Y) = 0$.

\section{Preliminaries}

\begin{defin}\label{D:GScycle}\cite[P.1, Equation~0.1]{GrossSchoen}
Let $X$ be a nice curve over a field $k$. Let $b$ be a $k$-rational point of $X$. Define the following diagonal cycles in $\CH^2(X \times X \times X)$:
\begin{alignat*}{3}\Delta_{12,b}&= \{(x,x,b):x\in X\},\quad \Delta_{23,b}&&= \{(b,x,x):x\in X\}, \quad
\Delta_{31,b}& &= \{(x,b,x):x\in X\}, \\
\Delta_{1,b}&=  \{(x,b,b):x\in X\},\quad \ 
\Delta_{2,b}&&=   \{(b,x,b):x\in X\},\quad  \
\Delta_{3,b}& &= \{(b,b,x):x\in X\}.
\end{alignat*}
Let $\Delta_{123} = \{ (x,x,x)\ | \ x \in X\}$.
The \defi{$b$-pointed modified diagonal cycle $\Delta_b$} is defined by
\[ \Delta_b \colonequals \Delta_{123}-\Delta_{12,b}-\Delta_{23,b}-\Delta_{13,b}+\Delta_{1,b}+\Delta_{2,b}+\Delta_{3,b}. \]
More generally, for a $k$-rational degree $1$ divisor $b \colonequals \sum a_i b_i$, define the \defi{$b$-pointed modified diagonal cycle $\Delta_b$} by $\Delta_b \colonequals \sum a_i \Delta_{b_i}$.
\end{defin}

\begin{defin}\label{D:PointedShadow}
Let $Z \in \CH^1(X \times X)$ denote a correspondence on a curve $X$. 
Let $\widetilde{Z} \colonequals Z \times X$. Let $b$ be a $k$-rational degree $1$ divisor. Define the \defi{$b$-pointed $Z$-shadow $\Sh(Z,b)$} to be
\[ \Sh(Z,b) \colonequals (\pi_3)_*(\Delta_b \cdot \widetilde{Z}) \in \CH^0(X),\]
where $\pi_3 \colon X \times X \times X \rightarrow X$ is the natural projection map onto the $3$rd coordinate.
\end{defin}

\begin{rem}\label{R:Shdeg0}
We note that in fact, $\Sh(Z,b)$ lies in the subgroup $\Pic^0(X)$ of $\Pic(X) = \CH^0(X)$ since $\Delta_b$ is homologically trivial. Furthermore, if $Z$ and $b$ are $k$-rational, so is $\Sh(Z,b)$.
\end{rem}

We think of the shadow as a witness to the nontriviality of $\Delta_b$:
\begin{lem}\label{L:Key}
Let $m$ be any integer. If $m\Sh(Z,b)$ is nonzero in $\CH^0(X)=\Pic(X)$ for some correspondence $Z$, then $m\Delta_b$ is nonzero in $\CH^2(X \times X \times X).$ 
\end{lem}

\begin{proof}  Immediate from the fact that the operations used to obtain $\Sh(Z,b)$ from $\Delta_b$ are all homomorphisms.
\end{proof}

The value of the shadow is that, in most cases, it is quite practical to compute with in a way that the modified diagonal cycle itself is not. We now derive a computable formula for the $Z$-shadow in terms of the fixed points of the correspondence $Z$.

\begin{defin}\label{D:NiceCorr}
We say a cycle $Z \in \CH^1(X \times X)$ is a \defi{nice} correspondence if none of the irreducible components in the support of $Z$ are vertical or horizontal fibers.  
\end{defin}

We denote by $\psi_i: Z \ra X$ the restriction of the projection $\pi_i$ to $Z \subset X \times X$. If $Z$ is a nice correspondence, then the $\psi_i$ are finite morphisms, and we denote by $d_Z$ and $e_Z$ the degrees $\deg \psi_1$ and $\deg \psi_2$, and by $f_Z$ and $g_Z$ the maps $(\psi_2)_* \psi_1^*$ and $(\psi_1)_* \psi_2^*$ from $\Pic(X)$ to $\Pic(X)$.  Note that $f_Z$ multiplies degrees of divisors in $\Pic(X)$ by $d_Z$ and $g_Z$ multiplies degrees by $e_Z$.  

\begin{lem}\label{L:ComputeSh}
Let $X$ be a nice curve over a field $k$. Let $\Delta \in \CH^1(X \times X)$ denote the diagonal cycle. Let $Z \in \CH^1(X \times X)$ denote a nice correspondence. Let $b \in \Pic(X)$ be a $k$-rational degree $1$ divisor. Then
\[\Sh(Z,b) = (\pi_1)_*(Z \cdot \Delta)  - f_Z(b) - g_Z(b) +[d_Z+e_Z - \deg(Z \cdot \Delta)]b. \]
\end{lem}

(Replacing $\pi_1$ by $\pi_2$ in this expression would not change the definition, since $Z \cdot \Delta$ is supported on the diagonal.)

\begin{proof}
This is a direct computation using the definition of $\Delta_b$ in Definition~\ref{D:GScycle}. To this end, let $\iota \colon \Delta \rightarrow X \times X \times X$ be the map defined by $\iota(x,x) \colonequals (x,x,x)$, and let $\iota_{12,b} \colon X \times X \rightarrow X \times X \times X$ be the map defined by $\iota_{12,b}(x_1,x_2) = (x_1,x_2,b)$. Define $\iota_{23,b} \colon X \times X \rightarrow X$ by $\iota_{23}(x_1,x_2) = (x_1,x_2,x_2)$ and $\iota_{13} \colon X \times X \rightarrow X \times X \times X$ by $\iota_{13}(x_1,x_2) = (x_1,x_2,x_1)$.

Then $\Delta_{123} \cdot \widetilde{Z} = \iota_*(Z \cdot \Delta),\Delta_{12,b} \cdot \widetilde{Z} = {\iota_{12,b}}_*(Z \cdot \Delta), \Delta_{23,b} \cdot \widetilde{Z} = {\iota_{23}}_*(\psi_1^*(b)), \Delta_{13,b} \cdot \widetilde{Z} = {\iota_{13}}_*(\psi_2^*(b))$, and $\Delta_{1,b} \cdot \widetilde{Z} = {\iota_{12,b}}_*(\psi_2^*(b)), \Delta_{2,b} \cdot \widetilde{Z} = {\iota_{12,b}}_*(\psi_1^*(b)), \Delta_{3,b} \cdot \widetilde{Z} = 0$. Since we also have $(\pi_3 \cdot i)|_\Delta = \pi_1 = \pi_2, (\pi_3 \cdot i_{23})|_Z = 
\psi_2, (\pi_3 \cdot i_{13})|_Z = \psi_1$ and $ \pi_3 \cdot i_{12,b}$ is the constant map sending all points to $b$, we get
\begin{align*}
(\pi_3)_*(\Delta_{123} \cdot \widetilde{Z}) &= (\pi_3 \cdot \iota)_*(Z \cdot \Delta) = (\pi_1)_*(Z \cdot \Delta) = (\pi_2)_*(Z \cdot \Delta) \\
(\pi_3)_*(\Delta_{12,b} \cdot \widetilde{Z}) &=  (\pi_3 \cdot \iota_{12,b})_*(Z \cdot \Delta) = \deg(Z \cdot \Delta) b  \\
(\pi_3)_*(\Delta_{23,b} \cdot \widetilde{Z}) &= (\pi_3 \cdot \iota_{23})_*(\psi_1^*(b)) = (\psi_2)_*(\psi_1^*(b)) = f_Z(b) \\
(\pi_3)_*(\Delta_{13,b} \cdot \widetilde{Z}) &= (\pi_3 \cdot \iota_{13})_*(\psi_2^*(b)) = (\psi_1)_*(\psi_2^*(b)) = g_Z(b)  \\
(\pi_3)_*(\Delta_{1,b} \cdot \widetilde{Z}) &= (\pi_3 \cdot \iota_{12,b})_*( \psi_2^*(b) ) = \deg(\psi_2) b = e_Z b \\
(\pi_3)_*(\Delta_{2,b} \cdot \widetilde{Z}) &= (\pi_3 \cdot \iota_{12,b})_*( \psi_1^*(b) ) = \deg(\psi_1) b = d_Z b  \\
(\pi_3)_*(\Delta_{3,b} \cdot \widetilde{Z}) &= (\pi_3)_*(0) = 0. \qedhere
\end{align*}
\end{proof}

\begin{rem}
In particular, if $Z = Z_f$ is the graph of a dominant self-morphism $f: X \ra X$, Lemma~\ref{L:ComputeSh} reduces to
\[
\Sh(Z_f,b) = (\pi_1)_* (Z_f \cdot \Delta) - f(b) - f^{-1}(b) + (\deg f + 1 - \deg(Z_f \cdot \Delta))b.
\]
In the still more particular case that $k = \ff_q$ and $f:X \ra X$ is the geometric Frobenius morphism, we have that $f(b) = b$ and $f^{-1}(b) = qb$ when $b$ is $\mathbb{F}_q$ rational, and the intersection of the graph of Frobenius with the diagonal is the divisor on $X$ obtained as the sum of all $\ff_q$-rational points.  So we get 
\begin{equation}\label{E:ShFr}
\Sh(Z_f,b) = X(\ff_q) - |X(\ff_q)|b,
\end{equation}
where by $X(\ff_q)$ we mean the divisor of degree $|X(\ff_q)|$ obtained by summing the $\ff_q$-rational points of $X$.
\end{rem}

\begin{defin} The \defi{$b$-pointed Frobenius shadow} on $X$ is the divisor class of $X(\ff_q) - |X(\ff_q)|b$.
\end{defin}

Another important special case is that where $f: X \ra X$ is the identity. In this case, the self-intersection of the diagonal becomes a copy of the anticanonical divisor on the diagonal, and we get
\begin{equation}\label{E:ShId}
    \Sh(Z_f,b) = (2g-2) b - K_X.    
\end{equation}

\begin{rem}\label{R:canbp} The fact that $\Delta_b$ is nontrivial if $(2g-2)b - K_X$ is nontrivial in $\Pic^0(X)$ is already observed, with the same proof, in \cite[Proposition 2.3.2.]{qiuzhang}. 
\end{rem}

The question of whether the modified diagonal cycle is torsion, which a priori depends on the choice of $b$, is somewhat rigidified by the above discussion; there exists a choice of $b$ making $\Delta_b$ torsion if and only if $\Delta_b$ is torsion for some (whence every) degree-$1$ divisor $b$ satisfying $(2g-2)b = K_X$.  We call such a $b$ a \defi{canonical base point}.  Of course, there may not be a $k$-rational canonical base point; so when we refer to a canonical base point, we silently extend the base field $k$ to an extension over which such a degree $1$ divisor $b$ is defined.

The following proposition provides a computationally tractable lower bound for the order of $\Delta_b$ for a curve $X$ over a finite field.

\begin{prop}\label{P:CanBP}\label{pr:frobshadownonzero}
Let $k = \ff_q$ and let $X$ be a nice curve over $k$. Let $b$ be a $k$-rational degree $1$ divisor. Then the order of $(2g-2)X(\ff_q) - |X(\ff_q)|K_X$ in $\Pic^0(X)$ divides the order of the $b$-pointed modified diagonal cycle of $X$. In particular, if $(2g-2)X(\ff_q) - |X(\ff_q)|K_X$ is nonzero in $\Pic^0(X)$, then the modified diagonal cycle $\Delta_b$ is nonzero for all $k$-rational degree-$1$ divisors $b$ on $X$. 
\end{prop}

\begin{proof}
Let $Z' \colonequals (2g-2)Z_{\Frob_q} - |X(\ff_q)| \Delta \in \CH^1(X \times X)$.  From the computations of the shadow of Frobenius \eqref{E:ShFr} and the shadow of the identity \eqref{E:ShId} above, one can check that the image of the modified diagonal cycle $\Delta_b$ under the group homomorphism from $\CH^2(X \times X \times X)$ to $\CH^0(X)$ that takes a cycle $W$ to $(\pi_3)_*(W \cdot (Z' \times X))$ is independent of $b$, and indeed is always equal to $(2g-2)X(\ff_q) - |X(\ff_q)|K_X$. The divisibility of orders follows. \qedhere
\end{proof}

\begin{defin}\label{D:FroSh}
Let $k = \ff_q$ and let $X$ be a nice curve over $k$. The {\defi{Frobenius shadow}} of $X$ is the point $(2g-2) X(\ff_q) - |X(\ff_q)| K_X$ in $\Pic^0(X)$. 
\end{defin}

\begin{rem} The use of ``shadows" of this kind to witness nontriviality of the modified diagonal cycle is not new; for instance, Bloch in essence uses the Frobenius shadow in \cite[\S~3]{Bloch1} in order to show the modified diagonal cycle of the Fermat quartic is nontrivial.  Darmon, Rotger, and Sols~\cite{DarmonRotgerSols} use the Hodge-theoretic version of the shadow (which takes image in $\Jac(X)(\mathbb{C})$) for any Hodge class in $X \times X$.  

Although most curves $X$ defined over number fields do not have any nontrivial Hodge classes in $X \times X$, they do acquire additional cycle classes upon reduction modulo $p$. The idea to use the cycle class of the Frobenius correspondence along with bounds on torsion on global Galois cohomology groups also appears in the work of Buhler, Schoen, and Top~\cite[Theorem~4.8]{BST} in the special case of certain genus $3$ curves embedded in a triple product of an elliptic curve. They in fact obtain the even stronger conclusion that the modified diagonal cycle is non-torsion modulo {\textit{algebraic}} equivalence in certain cases \cite[Lemma~2.2, Lemma~2.6, Lemma~4.8~(3)]{BST}, by showing that imposing a certain irreducibility hypothesis on the relevant Galois module allows the refined cycle class map to factor through the Griffiths group.

The novelty of the present paper is to use shadows corresponding to all the correspondences on $X$ over finite fields, and to combine these with bounds from weights in order to systematically prove non-torsionness of the modified diagonal cycle for a large population of curves. 
\end{rem}

\subsection{The $\ell$-adic Abel-Jacobi image of the modified diagonal cycle}

The algorithms we will present later in the paper rely on the fact that the shadows above can also be understood through the lens of \textup{\'{e}t}ale cohomology. Let $W$ be a smooth variety over a field $k$.  Let $\Zhat' \colonequals \prod_{\ell \neq \chara(k)} \zz_\ell$. Let $r$ be an integer $\geq 0$. By \cite[Corollary~1.2]{Bloch1} the space of null-homologous codimension $r$ cycles admits a cycle class map
\begin{equation}\label{E:Bloch}
\Sigma_{W}: \homtriv{W}{r} \ra H^1(G_k, H^{2r-1}_{\textup{\'{e}t}}(W_{\overline{k}}, \Zhat'(r))).
\end{equation}

Recall that $\Delta_b$ lies in $\homtriv{X \times X \times X}{2}$.
\begin{defin}\label{D:CerClass} The \defi{$b$-pointed modified diagonal class} $\nu(X,b)$ is $\Sigma_{X \times X \times X}(\Delta_b) \in H^1(G_k, H^3_{\textup{\'{e}t}}(X_{\overline{k}} \times X_{\overline{k}} \times X_{\overline{k}}, \Zhat'(2))) $. For $\ell \nmid \chara k$, the \defi{$\ell$-adic $b$-pointed modified diagonal class} $\nu_\ell(X,b)$ is the projection of $\nu(X,b)$ to $H^3_{\textup{\'{e}t}}(X_{\overline{k}} \times X_{\overline{k}} \times X_{\overline{k}}, \Z_\ell(2)))$. 

The \textup{\'{e}t}ale cohomology groups of $X^3$ have a K\"{u}nneth decomposition. In particular, 
we have
\[ H^3_{\textup{\'{e}t}}(X_{\overline{k}} \times X_{\overline{k}} \times X_{\overline{k}}, \Zhat') = \oplus_{\substack{0 \leq i,j,k \leq 2 \\ i+j+k=3}} H^i_{\textup{\'{e}t}}(X_{\overline{k}}, \Zhat') \otimes H^j_{\textup{\'{e}t}}(X_{\overline{k}}, \Zhat') \otimes H^k_{\textup{\'{e}t}}(X_{\overline{k}}, \Zhat'). \] (We have suppressed the pull-back by the projection maps $X^3 \rightarrow X$.) 

\begin{rem}\label{R:wheredoesMDlive}
It follows from \cite[Corollary 2.6]{GrossSchoen} that $\nu(X,b)$ in fact lies in the single K\"{u}nneth component $H^1(G_k, H^1_{\textup{\'{e}t}}(X_{\overline{k}}, \Zhat') \otimes H^1_{\textup{\'{e}t}}(X_{\overline{k}}, \Zhat') \otimes H^1_{\textup{\'{e}t}}(X_{\overline{k}}, \Zhat')(2))$.
\end{rem}

The procedure which sends a class in $\homtriv{X \times X \times X}{2}$ to its shadow can be carried out within these \textup{\'{e}t}ale cohomology groups---intersecting a homologically trivial cycle with the cycle $Z \times X$ for a correspondence $Z$ inside $X \times X$ translates into taking the cup product of the corresponding cycle classes. More precisely, we have

\begin{prop}\label{P:IntersectToCup}
The following diagram commutes, where the top right arrow is the shadow construction, and the bottom arrow is given by cup product in Galois cohomology.
\begin{equation*}\label{}
  \xymatrix{
    \homtriv{X \times X \times X}{2} \times  \CH^1(X \times X) \ar[d]_{(\Sigma_{X \times X \times X},\Sigma_{X \times X})} \ar[r]^-{(\pi_3)_* \left( (\cdot) \cap \left(  (\cdot) \times X  \right)\right)} & \homtriv{X}{1}=\Pic^0(X) \ar[d]^{\Sigma_X} \\
      H^1(G_k, H^3_{\textup{\'{e}t}}(X_{\overline{k}} \times X_{\overline{k}} \times X_{\overline{k}}, \Zhat'(2))) \times  H^0(G_k, H^2_{\textup{\'{e}t}}(X_{\overline{k}} \times X_{\overline{k}}, \Zhat'(1))) \ar[r]^-{\phi \circ \cup} & H^1(G_k, H^1_{\textup{\'{e}t}}(X_{\overline{k}}, \Zhat'(1))), \\
  }
    \end{equation*}
    where $\phi$ is defined (using the K\"{u}nneth decomposition) by
    \begin{align*}
H^3_{\textup{\'{e}t}}(X_{\overline{k}} \times X_{\overline{k}} \times X_{\overline{k}}, \Zhat'(2)) \times H^2_{\textup{\'{e}t}}(X_{\overline{k}} \times X_{\overline{k}}, \Zhat'(1)) &\rightarrow H^1_{\textup{\'{e}t}}(X_{\overline{k}}, \Zhat'(1)) \\
(h_1 \otimes h_2 \otimes h_3), (h_4 \otimes h_5) &\mapsto (h_1 \cup h_4) (h_2 \cup h_5) h_3.
    \end{align*}
\end{prop}
\begin{proof}
This is \cite[Proposition~1.9~(6)]{BST}, taking $W = X \times X \times X$, $W' = X$, and $\Gamma = Z \times \Delta \in X \times X \times X \times X$.
\end{proof}

\begin{rem}\label{R:IgnoreH0H2} Since $\nu(X,b)$ lies in $H^1(G_k, H^1_{\textup{\'{e}t}}(X_{\overline{k}}, \Zhat') \otimes H^1_{\textup{\'{e}t}}(X_{\overline{k}}, \Zhat') \otimes H^1_{\textup{\'{e}t}}(X_{\overline{k}}, \Zhat')(2))$,  Proposition~\ref{P:IntersectToCup} implies that the $\ell$-adic cohomology class of $\Sh(Z,b)$ depends only on the projection of $Z$ to the K\"{u}nneth component $H^1_{\textup{\'{e}t}}(X_{\overline{k}}, \Zhat') \otimes H^1_{\textup{\'{e}t}} (X_{\overline{k}}, \Zhat')$.
\end{rem}

\begin{rem}\label{R:CycKummer}
Observe that Bloch's map $\Sigma_X \colon \Pic^0(X) \rightarrow H^1(G_k,H^1_{et}(X,\Zhat'(1)))$ defined using the Hochschild-Serre spectral sequence is the Kummer map. (See, e.g., the proof of \cite[Corollary~1.2]{Bloch1} for details.)
\end{rem}

Cycle class maps are also compatible with specialization modulo $p$. More precisely, we have
\begin{prop}\label{P:reduce}
Let $W$ be a smooth, projective geometrically integral variety over a number field $K$. Let $\mathfrak{p}$ be a prime of good reduction for $W$. Let $\ff_q$ denote the residue field of $\mathfrak{p}$, and let $\Zhat' = \prod_{\ell \neq \chara(\ff_q)} \zz_\ell$. Assume that $H^{2r-1}(W_{\overline{K_\mathfrak{p}}},\Zhat'(r))$ is torsion free, and let $\theta \colon H^{2r-1}(W_{\overline{\ff_q}},\Zhat'(r)) \rightarrow H^{2r-1}(W_{\overline{K_\mathfrak{p}}},\Zhat'(r))$ denote the base change isomorphism in \cite[VI.4.2]{Milne}.
Then the following diagram commutes
\begin{equation*}
\xymatrix{
\homtriv{W_K}{r} \ar[r] \ar[d]_{\Sigma} & \homtriv{W_{K_{\mathfrak{p}}}}{r} \ar[r]^{\mathrm{sp}} \ar[d]_{\Sigma} & \homtriv{W_{\ff_q}}{r} \ar[d]_{\Sigma} \\
H^1(G_K,H^{2r-1}_{et}(W_{\overline{K}},\Zhat'(r))) \ar[r] & H^1(G_{K_\mathfrak{p}},H^{2r-1}_{et}(W_{\overline{K_{\mathfrak{p}}}},\Zhat'(r))) \ar[r]^{\theta^{-1}}
 & H^1(G_{\ff_q},H^{2r-1}_{et}(W_{\overline{\ff_q}},\Zhat'(r))) }
\end{equation*}
\end{prop}
\begin{proof}
This is \cite[Proposition~1.9~(5), Section~3]{BST}.
\end{proof}

\begin{rem}\label{R:hypsat}
The hypothesis on $W$ is satisfied whenever $W=X^n$ for a curve $X$ since the cohomology of $X$ is torsion-free and the Kunneth formula holds for \textup{\'{e}t}ale cohomology.
\end{rem}
\end{defin}

\begin{rem}\label{R:MDCeresa}(Relationship between modified diagonal and Ceresa classes) 
 There is a close relationship between the $b$-pointed modified diagonal class and the Ceresa class of an algebraic curve, whose definition we now recall. Let $i_b \colon X \rightarrow J$ be the Abel-Jacobi map sending $x$ to the class of the degree $0$ divisor $x-b$. The {\defi{Ceresa class $\Cer(X,b)$}} is the cycle class in $H^1(G_k, H^{2g-3}_{et}(J_{\overline{k}}, \Zhat'(g-1)))$ of the homologically trivial $1$-cycle $i_b(X) - [-1]^*i_b(X)$ on $J$. Using the natural isomorphisms of $G_k$-modules $\Lambda^{2g} H^1(X_{\overline{k}}) \cong \Zhat'(1)$ and $H^1(X_{\overline{k}}) \cong H_1(X_{\overline{k}})(-1)$ (Poincar\'e duality), we get a map of coefficient modules $(H^1(X_{\overline{k}}))^{\otimes 3}(2) \rightarrow \Lambda^3 H_1(-1) \cong \Lambda^{2g-3} H^1(X_{\overline{k}}) \cong H^{2g-3}(J_{\overline{k}})$, so we may also view the Ceresa class $\Cer(X,b)$ as a class in $H^1(G_k,\Lambda^3 H_1(-1))$. Let $\omega \in \Lambda^2 H_1(-1)$ be the element corresponding to the intersection pairing. 

Let $i_b^{(3)} \colon X \times X \times X \rightarrow J$ be the map sending $(x,y,z)$ to the class of the degree $0$ divisor $x+y+z-3b$. A direct computation shows that ${i_b^{(3)}}_*(\Delta_b)$ is the cycle $[3]_*(i_b(C))-3 [2]_*(i_b(C))+3i_b(C)$ on $J$. Combining this with the $\ell$-adic analogue of \cite[Proposition~2.9]{ColomboVG} shows that $\nu(X,b)$ maps to  $3\Cer(X,b)$ under the induced map on Galois cohomology groups from the K\"{u}nneth projector. 

At the level of cycles, $\Delta_b$ and $i_b(X) - [-1]^*i_b(X)$ belong to different Chow groups, so it does not make sense to ask whether one is a multiple of the other.  However, it is still true that one cycle is torsion in Chow if and only if the other is~\cite[Thm 1.5.5]{ZhangGS}.

\end{rem}

\section{Certifying nontriviality of modified diagonal cycles}
For the rest of this section $K$ will denote a number field. Let $X$ be a nice curve over $K$. Let $b \in X(K)$. Let $\nu_\ell(X,b)$ be the $\ell$-adic $b$-pointed modified diagonal class of $X$. In \S~\ref{S:Algo}, we give an algorithm which, when it terminates, provably  certifies that the modified diagonal cycle of $X$ has infinite order in $\CH^{2}(X \times X \times X)$. The algorithm combines upper bounds on the torsion order of the modified diagonal class (Proposition~\ref{P:DivofTorsOrd}) obtained using the theory of weights, together with lower bounds obtained from the order of the Frobenius shadow at a good prime $\mathfrak{p}$. In \S~\ref{S:Chebotarev}, we investigate the question of whether the algorithm would always produce a certificate of nontriviality of a nontorsion modified diagonal class if allowed to run long enough.

 \subsection{An algorithm}\label{S:Algo}
 \begin{prop}\label{P:DivofTorsOrd}
Let $M$ be the $G_K$ module  $H^3_{\textup{\'{e}t}}(X_{\overline{K}} \times X_{\overline{K}} \times X_{\overline{K}}, \Z_\ell(2))$.
Let $\mathfrak{p}$ be a prime ideal of $\mathcal{O}_K$ of good reduction for $X$ coprime to $\ell$, and let $N_{\mathfrak{p}} \colonequals \det(\Frob_{\mathfrak{p}}-I)|_M$. Then $N_{\mathfrak{p}}$ is nonzero and independent of $\ell$. Furthermore, if $\nu_\ell(X,b)$ is torsion of order $n$, then $n$ divides $N_{\mathfrak{p}}$.
 \end{prop}
\begin{proof}
Since $\mathfrak{p}$ is a prime of good reduction for $X$, and since $\ell$ is coprime to $\mathfrak{p}$, the weights for the $G_K$ action on $H^3_{\textup{\'{e}t}}(X_{\overline{K}} \times X_{\overline{K}} \times X_{\overline{K}}, \Z_\ell(2))$ are $-1$. Furthermore, the characteristic polynomial of $\Frob_{\mathfrak{p}}$ has integer coefficients and is independent of $\ell$. In particular $1$ is not an eigenvalue for $\Frob_{\mathfrak{p}}$ acting on $M$ and it follows that $N_{\mathfrak{p}}$ is a nonzero integer. 

We have a surjection $H^0(G_K,M/nM) \rightarrow H^1(G_K,M)[n]$ arising from the short exact sequence of $G_K$-modules $0 \rightarrow M \xrightarrow{\times n} M \rightarrow M/nM \rightarrow 0$.  The hypotheses of the proposition guarantee the existence of an element of exact order $n$ in $H^1(G_K,M)[n]$ and whence also in $H^0(G_K,M/nM)$.  This latter element $v$ is fixed by $G_K$ and in particular by $\Frob_{\mathfrak{p}}$.  So $\Frob_{\mathfrak{p}} - 1$ is an endomorphism of the free $\Z/n\Z$-module $M/nM$ which kills an element $v$ of exact order $n$; it follows that $n | \det \Frob_{\mathfrak{p}} - 1$, which was to be proved.
 \end{proof}

\begin{cor}\label{C:DivofOrd}
With notation as in Proposition~\ref{P:DivofTorsOrd}, if $\nu(X,b)$ is torsion of order $n$, then $n$ divides $p^k N_{\mathfrak{p}}$ for some $k \geq 0$.  \end{cor}
\begin{proof} Proposition~\ref{P:DivofTorsOrd} shows that the order of $\nu_\ell(X,b)$ divides $N_{\mathfrak{p}}$ for all $\ell \neq p$.
\end{proof}

Recall the definition of the Frobenius shadow of a curve over a finite field from Definition~\ref{D:FroSh}.

\begin{prop}\label{P:FroShDiv}
Let $\mathfrak{p}'$ be a prime ideal of $\mathcal{O}_K$ of good reduction for $X$. Let $M_{\mathfrak{p}'}$ be the order of the Frobenius shadow of the reduction of $X$ modulo $\mathfrak{p}'$. If $\nu(X,b)$ is torsion of order $n$, then $M_{\mathfrak{p}'}$ divides $(p'^{k'}) n$ for some $k' \geq 0$. 
\end{prop}
\begin{proof}
Let $\mathbb{F}$ denote the residue field at $\mathfrak{p}'$. Let $\kappa_{\mathfrak{p}'}$ denote the image of the Frobenius shadow at $\mathfrak{p}'$ under  the Kummer map $\kappa \colon \Pic^0(X_{\mathbb{F}}) \rightarrow H^1(G_{\mathbb{F}},H^1_{\textup{\'{e}t}}(X_{\mathbb{F}}, \Zhat'(1)))$. Since $\kappa$ is injective on the prime-to-$p'$ torsion of $\Pic^0(X_{\mathbb{F}})$, we see that $M_{\mathfrak{p}'}$ divides $p'^{k'} \ord(\kappa_{\mathfrak{p}'})$. 

Proposition~\ref{P:CanBP}, Proposition~\ref{P:reduce} and Remark~\ref{R:hypsat} with $W=X^3, r = 2$ and Proposition~\ref{P:IntersectToCup} imply that $\kappa_{p'}$ is the image of $\nu(X,b)$ under a group homomorphism, and hence $\ord(\kappa_{\mathfrak{p}'})$
divides the order of $\nu(X,b)$. Combining the last two sentences gives the result. 
\end{proof}

\begin{prop}\label{P:CAlgo} Let $\mathfrak{p},\mathfrak{p}'$ be prime ideals of $\mathcal{O}_K$ of good reduction for $X$. Let $M_{\mathfrak{p}'}$ be the order of the Frobenius shadow of the reduction of $X$ modulo $\mathfrak{p}'$. Let $N_{\mathfrak{p}}$ be as in Proposition~\ref{P:DivofTorsOrd}. If $(\mathfrak{p}')^{-k'}M_{\mathfrak{p}'} \nmid p^kN_{\mathfrak{p}}$ for every $k,k' \geq 0$, then $\nu(X,b)$, and a fortiori the modified diagonal cycle of $X$, has infinite order.
\end{prop}
\begin{proof}  
Combine Corollary~\ref{C:DivofOrd} and Proposition~\ref{P:FroShDiv}.
\end{proof}

The following algorithm thus provably certifies nontriviality of $\nu(X,b)$ and therefore also of the modified diagonal cycle $\Delta_b$.

\begin{algo}\label{alg:upperbound} \hfill
\begin{enumerate}[\upshape (1)]
\item Choose a nonempty finite set \(\mathcal{T}\) of auxiliary good primes of $X$.
\item For each \(\mathfrak{p}\) in $T$, compute the integer $N_{\mathfrak{p}} \colonequals \det(\Frob_{\mathfrak{p}}-I)|_M$, where $M$ is the $G_K$-module in Proposition~\ref{P:DivofTorsOrd}, and the order $M_{\mathfrak{p}}$ of the Frobenius shadow of the reduction of $X$ modulo $\mathfrak{p}$.

\item Let \(\widetilde{N} = \gcd_{\mathfrak{p} \in \mathcal{T}}(p^{\infty}N_{\mathfrak{p}})\), and let \(\widetilde{M} = \lcm_{p \in \mathcal{T}}(p^{-\infty}M_{\mathfrak{p}})\).
\end{enumerate}
If $\widetilde{M} \nmid \widetilde{N}$, then output that the modified diagonal cycle of $X$ has infinite order.
\end{algo}
 \begin{rem}\label{R:CerTorGeom}The property of $\nu_\ell(X,b)$ being torsion is a geometric property of $X$, i.e., can be checked after a finite base extension of the ground field. More precisely, let $K$ be any field, and let $L$ be a finite extension. Let $X$ be a nice curve over $K$, and let $b \in X(K)$. Let $\ell \neq \chara(K)$ be a prime. Then $\nu_\ell(X,b)$ is torsion if and only if $\nu_\ell(X_L,b)$ is torsion.
 This follows from the restriction-corestriction sequence for Galois cohomology---see, e.g., \cite[Proposition~2.24]{BLLS}. 
\end{rem}

\subsection{Chebotarev}\label{S:Chebotarev}

We now turn to the question complementary to the one treated above, namely: how might we certify that the modified diagonal class of a curve $X/K$ {\em is} torsion?  For instance, suppose we attempt to use the algorithm above to show that a curve $X$ over a number field $K$ has nontorsion $\nu(X,b)$; we compute $N_{\mathfrak{p}}$ for prime after prime, only to find that rather than growing this lower bound is fixed at, say, $12$.  Is there some finite amount of data of this kind that proves that $\nu(X,b)$ is, in fact, torsion? 

In this direction, we prove the following theorem, which shows that generic curves with non-torsion $\nu(X,b)$ have reductions mod $p$ whose $\ell$-adic classes $\nu_\ell(X_{\mathbb{F}_\mathfrak{p}},b)$ have unbounded orders.  The main tool is the Chebotarev density theorem, which means that with some effort (and some tolerance for pretty bad bounds) the result could be made effective.

\begin{notation} Denote $H^1_{et}(\bar{X},\Z_\ell)$ by $H$.  Replacing $K$ by a finite extension if necessary, take $b$ to be the canonical base point, i.e., a degree $1$ divisor such that $(2g-2)b$ is linearly equivalent to $K_C$. Let $\Delta_b$ be the corresponding modified diagonal; then, as we observed in  Remark~\ref{R:wheredoesMDlive}, the image $\nu_\ell(X,b)$ of $\Delta_b$ under the cycle class map to \textup{\'{e}t}ale cohomology lies in $H^1(G_K, (H \otimes H \otimes H)(2)).$ 
\end{notation}

\begin{thm}\label{T:shadowcheb}
Let $X$ be a curve of genus $g \geq 3$ over a number field $K$ and let $\ell$ be a prime greater than $3$.  Suppose the image of $G_K$ in the automorphism group of $H_1(X_{\bar{K}}, \Z_\ell)$ has finite index in $\GSp_{2g}(\Z_\ell)$, and suppose furthermore that $\nu_\ell(X,b)$ is not torsion.  Then, for every integer $k \geq 0$, there is a prime $\mathfrak{p}$ of $K$ (prime to $\ell$ and of good reduction for $X$) such that the order of the Frobenius shadow in $\Pic^0(X)(\mathbb{F}_\mathfrak{p})$ is divisible by $\ell^k$.
\end{thm}

Before proving this proposition, we gather some preparatory definitions and lemmas. In Definition~\ref{D:Grpthyshadow}, we describe how to attach a shadow to an arbitrary element of $G_K$ purely group-theoretically, generalizing the definition of the Frobenius shadow. Under the hypotheses that $G_K$ has big image on $H$ and that $\nu(X,b)$ has infinite order, we show that any element $\sigma_0$ of $G_K$ may be perturbed to an element $\sigma$ which acts in the same way on $H$, but whose group-theoretic shadow has infinite order (Proposition~\ref{P:JhasBigImage} and Proposition~\ref{P:Maketargetlarge}). Since $\Sh_{\sigma}$ varies continuously with $\sigma$ (Lemma~\ref{L:ContinuitySh}), we apply the Chebotarev density theorem to argue that for any integer $k$, there exists a $\Frob_p$ close to $\sigma$ as above in the profinite topology whose Frobenius shadow has order divisible by $\ell^k$. 

Let $\chi$ denote the $\ell$-adic cyclotomic character of $G_K$. For any $\sigma \in G_K$, the action of $\sigma$ affords an endomorphism of $H$, which we may think of as an element $\Gamma_\sigma$ of $(H \otimes H)(1)$ by means of the symplectic form on $H$.  Note that the Tate twist comes from the fact that, while $H$ is isomorphic to its dual as an abelian group, this isomorphism cannot in general be made $\sigma$-equivariant; rather, the dual of $H$ is naturally isomorphic to $H(1)$ -- here $H(1)$ denotes the $\sigma$-module on which $\sigma$ acts via $\chi(\sigma)\sigma$.
\begin{defin}\label{D:Grpthyshadow} Let $\sigma \in G_K$.  The {\defi{group-theoretic shadow}} $\Sh_\sigma$ of $\sigma \in G_K$ is obtained by pairing the restriction of $\nu_\ell(X,b)$ to $H^1(\langle \sigma \rangle, (H \otimes H \otimes H)(2))$ with $\Gamma_\sigma \in H^0(\langle \sigma \rangle, (H \otimes H)(1))$, again using the symplectic form on $H$ to contract the first two factors of the coefficients of $\nu_\ell(X,b)$ with the coefficients of $\Gamma_\sigma$.  The resulting product $\nu(X,b) | \langle \sigma \rangle \cup \Gamma_\sigma$ then lies in $H^1(\langle \sigma \rangle, H(1)) \cong H / (\sigma \chi(\sigma) - 1) H.$ 
\end{defin}

\begin{rem}\label{R:ShadowsAgree}
This definition is meant to generalize the notion of Frobenius shadow for curves over finite fields.  Recalling Proposition~\ref{P:IntersectToCup} and the discussion following, we observe that when $\mathfrak{p}$ is a prime of good reduction for $X$ which is prime to $\ell$, we have that $(2g-2)\Sh_{\Frob_\mathfrak{p}}$ in $G_K$ is the $\ell$-adic Kummer class of the Frobenius shadow $\Sh(\Frob_\mathfrak{p}, b)$ we have already defined in $\Pic^0(X)(\mathbb{F}_\mathfrak{p})$.
\end{rem}

\begin{lem}\label{L:ContinuitySh}
For every nonnegative integer $k$, the set $\Sigma_k$ of $\sigma \in G_L$ such that $\ell^k \Sh_{\sigma} = 0$ is closed in the profinite topology.
\end{lem}

\begin{proof}
Let $\zeta$ be a continuous $1$-cycle from $G_K$ to $H \otimes H \otimes H(2)$ representing the class $c_b$.
To say $\ell^k\Sh_{\sigma} = 0$ is to say that the pairing of $\ell^k \zeta(\sigma)$ with $\Gamma_\sigma$ in $H$ lies in the subspace $(\sigma \chi(\sigma) -1) H$.  The desired statement follows from the fact that $\zeta(\sigma)$, $\Gamma_\sigma$, and the subspace $(\sigma \chi(\sigma) -1)H$ all vary continuously with $\sigma$.  
\end{proof}

\begin{notation}\label{N:J} We will use the following notation for the remainder of this section.
\begin{itemize}
\item Let $L/K$ be the minimal extension such that $G_L$ acts trivially on $H$. Then $\nu(X,b)$ restricts to an element of $H^1(G_L, H \otimes H \otimes H(2))$, which is to say a homomorphism $J: G_L \ra H \otimes H \otimes H(2)$. 

\item Let $\wedge^3 H$ be the subspace of $H \otimes H \otimes H(2)$ consisting of those elements transforming as the sign character under the natural action of $S_3$. 
We write $h_1 \wedge h_2 \wedge h_3 \in \wedge^3 H$ to mean the element of $H \otimes H \otimes H$ of the form $\sum_{\pi \in S_3} \sgn(\pi) h_{\pi(1)} \otimes h_{\pi(2)} \otimes h_{\pi(3)}.$

\item The representation $\wedge^3 H$ of $\GSp$ has a natural subrepresentation $\wedge_0^3 H$, namely the kernel of the restriction to $\wedge^3 H$ of the  map $H^{\otimes 3}(2) \rightarrow H(1)$ sending $h_1 \otimes h_2 \otimes h_3$ to $\omega(h_1, h_2) h_3$.

\item Let $e_1, f_1, \ldots, e_g, f_g$ be a set of standard symplectic generators of $H$.
\end{itemize}
\end{notation}

\begin{prop}\label{P:wedge0}
Let $\ell > 3$. The class $\nu_\ell(X,b) \in H^1(G_K, H \otimes H \otimes H(2))$ in fact lies in the image of $H^1(G_K, \wedge^3_0 H)$.
\end{prop}
\begin{proof}
The invariance of $\Delta_b$ under the action of $S_3$ on the three copies of $X$ implies that $\nu_\ell(X,b) \in H^1(G_K, H \otimes H \otimes H(2))^{S_3} \cong H^1(G_K, (H \otimes H \otimes H(2))^{S_3}) = H^1(G_K,\wedge^3 H)$ since $\ell > 3$. 

We now argue that in fact $\nu_\ell(X,b)$ lies in the image of $H^1(G_K, \wedge^3_0 H)$.
We have already computed that the class in $H^1(G_K,H(1))$ obtained by pairing the class of $\Delta_b$ in $H^1(G_K, H \otimes H \otimes H(2))$ with the class of the diagonal in $H^0(G_K, H \otimes H(1))$ is the shadow of the identity map, which is $(2g-2)b - K_X$ by \eqref{E:ShId}; by our choice of $b$, this is zero. Since $\ell > 3$, a class in the kernel of the map $H^1(G_K,\wedge^3 H) \rightarrow H^1(G_K,H(1))$ has a unique preimage in $H^1(G_K,\wedge^3_0 H)$. \qedhere
    \end{proof}
    
\begin{prop}\label{P:JhasBigImage} Suppose that $\nu_\ell(X,b)$ is not torsion and that $b$ is the canonical base point. Then the image of $J$ is a finite index subgroup of $\wedge^3_0 H$. In particular, there exists an element $\tau \in G_L$ such that $J(\tau) = \ell^\alpha e_1 \wedge e_2 \wedge e_3$ for some non-negative integer $\alpha$.
 \end{prop}
 
\begin{proof} Since $J$ is a restriction of $\nu_\ell(X,b)$, its image lies in the subspace $\wedge^3_0 H$ by Proposition~\ref{P:wedge0}.
The fact that the image of $J$ is {\em large} in $\wedge^3_0 H$ is more difficult, and follows from the main theorem proved by Hain and Matsumoto in \cite{hainmatsumoto}.  More specifically:  $H^1(G_K, \wedge^3_0 H)$ is naturally identified with the Galois cohomology group $H^1(G_K, \wedge^3 H / H)$, which in the notation of \cite{hainmatsumoto} is called $H^1(G_K, L_{\Z_\ell}/H_{\Z_\ell})$.  Now the large monodromy hypothesis for the action of $G_K$ on $H$ and the hypothesis that the Ceresa class of $X$ has infinite order in $H^1(G_K, L_{\Z_\ell}/H_{\Z_\ell})$ guarantee that condition (c) of \cite[Theorem 10.6]{hainmatsumoto} is satisfied, which implies that condition (b) is satisfied, which implies that the image of $J$ has finite image.  Indeed, the index of the image of $J$ in $\wedge^3_0 H$ must be a power of $\ell$.  Note also that $e_1 \wedge e_2 \wedge e_3$ lies in $\wedge^3_0 H \subset \wedge^3 H$ because $\omega(e_i, e_j) = 0$ for all $i,j$. So some $\ell$-power multiple of $e_1 \wedge e_2 \wedge e_3$ must be contained in the image of $J$, as claimed.\qedhere
\end{proof}

\begin{lem}\label{L:ShPerturb}
Suppose $\sigma \in G_K$ and $\tau \in G_L$. Under the identification $H/(\tau \sigma \chi(\sigma) -1)H \cong H/(\sigma \chi(\sigma)-1)H$ induced by the identity map on $H$, we have
\[\Sh_{\tau \sigma} = J(\tau) \cup \Gamma_{\sigma} + \Sh_\sigma.\]
\end{lem}
\begin{proof}
If we represent $\nu_\ell(X,b)$ by a $1$-cocyle $\zeta: G_K \ra H \otimes H \otimes H$, then \begin{equation*}
    \zeta(\tau \sigma) = \zeta(\tau) \zeta(\sigma)^\tau = J(\tau) \zeta(\sigma)^\tau.
\end{equation*}
Since $\tau \in G_L$, we have $\zeta(\sigma)^\tau = \zeta(\sigma)$ and $\Gamma_{\tau \sigma} = \Gamma_{\sigma}$ and it follows that
\begin{equation*}
    \Sh_{\tau \sigma}  = (J(\tau) + \zeta(\sigma)^\tau) \cup \Gamma_{\tau \sigma} = (J(\tau) + \zeta(\sigma)) \cup \Gamma_{\sigma} = J(\tau) \cup \Gamma_{\sigma} + \Sh_\sigma. \qedhere
\end{equation*}
\end{proof}

\begin{prop}\label{P:Maketargetlarge} Suppose $\tau \in G_L$ is such that $J(\tau) = \ell^{\alpha} e_1 \wedge e_2 \wedge e_3$ for some integer $\alpha$.
If the image of $G_K$ acting on $H$ has finite index in $\GSp_{2g}(\mathbb{Z}_\ell)$, there exists an element $\sigma \in G_K$ such that 
\begin{enumerate}[\upshape (i)]
\item $e_1$ has infinite order in $H/(\sigma \chi(\sigma) -1)H$, and, \item $\Sh_{\tau \sigma} = \Sh_{\sigma} + ce_1$ for some nonzero constant $c$. 
\end{enumerate}
In particular, for such choice of $\tau$ and $\sigma$, if  $\Sh_{\sigma}$ has finite order, then $\Sh_{\tau \sigma}$ has infinite order.
\end{prop}
\begin{proof}
By our hypothesis that the image of $G_K$ is open in $\GSp_{2g}(\Z_\ell)$, we can choose a $\sigma \in G_K$ which has $\chi(\sigma) = 1$, which fixes $e_1$ and $f_1$, and which preserves the space generated by $e_2,f_2,e_3,f_3$, acting on that space via some matrix $B \in \Sp_4(\Z_\ell)$.
For this choice of $\sigma$, the group $H / (\sigma - 1) H$ is infinite and the image of $e_1$ in this group has infinite order.

By Lemma~\ref{L:ShPerturb}, we know $\Sh_{\tau \sigma} = J(\tau) \cup \Gamma_{\sigma} + \Sh_{\sigma}$. The term $J(\tau) \cup \Gamma_{\sigma}$ can be computed explicitly. Recall that the pairing of $h_1 \otimes h_2 \otimes h_3$ with $\Gamma_\sigma$ is $\langle h_1, \sigma h_2 \rangle h_3$. 
So $J(\tau) \cup \Gamma_{\sigma}$ is the sum of six terms
\begin{equation*}
    e_1 \omega(e_2,\sigma e_3) + e_2 \omega(e_3,\sigma e_1) + e_3 \omega(e_1,\sigma e_2) - e_1\omega(e_3, \sigma e_2) - e_2\omega(e_1, \sigma e_3) - e_3\omega(e_2, \sigma e_1).
\end{equation*}
The second, third, fifth, and sixth terms are all zero.  The first term is $e_1\omega(e_2, B e_3)$ and the fourth is $-e_1\omega(e_3, B e_2)$. The big monodromy hypothesis on $G_K$ means we are free to choose $B$ to be a $4 \times 4$ generalized symplectic matrix whose $(e_3, f_2)$ and $(e_2, f_3)$ entries are not equal.  By doing so, we have guaranteed that $J(\tau) \cup \Gamma_{\sigma}$ is a nonzero multiple of $e_1$. 
\end{proof}

\begin{proof}[Proof of Theorem~\ref{T:shadowcheb}]
Proposition~\ref{P:JhasBigImage} and Proposition~\ref{P:Maketargetlarge} together show that there exists $\sigma \in G_K$ such that $\Sh_\sigma$ has infinite order. (These use the hypothesis that $G_K$ has big monodromy and the modified diagonal class has infinite order.) A continuity argument now finishes the proof, as we now explain. 

Let $\Sigma_k$ be the collection of $\sigma \in G_L$ such that $\ell^k \Sh_{\sigma} = 0$ as in Lemma~\ref{L:ContinuitySh}. The Chebotarev density theorem guarantees that the set of Frobenii is dense in $G_K$.  In particular, since $\Sh_{\sigma}$ has infinite order, it does not lie in $\Sigma_k$, so $\sigma$ has an open neighborhood disjoint from $\Sigma_k$ by Lemma~\ref{L:ContinuitySh}; this neighborhood must contain $\Frob_{\mathfrak{p}}$ for some prime $\mathfrak{p}$ of $K$; and the group-theoretic shadow of Frobenius for this prime is therefore not killed by $\ell^k$. Now $2g-2$ times the group-theoretic shadow of Frobenius $\Sh_{\Frob_p}$ agrees with the $\ell$-adic Kummer class of the Frobenius shadow $\Sh(\Frob_\mathfrak{p}, b) \in \Pic^0(X)(\F_p)$
(Remark~\ref{R:ShadowsAgree}).  Suppose $\ell^a N \Sh(\Frob_\mathfrak{p}, b) = 0$ for some integer $N$ prime to $\ell$.  Then $\ell^a$ kills the $\ell$-adic Kummer class of $\Sh(\Frob_\mathfrak{p}, b)$, which by the above discussion implies that $\ell^a$ kills $(2g-2)\Sh_{\Frob_p}$, so $a$ cannot be any less than $k - \ord_\ell (2g-2)$.  In particular, the multiplicative order of the point $\Sh(\Frob_\mathfrak{p}, b)$ can be made as $\ell$-divisible as we like by a suitable choice of $\mathfrak{p}$.
\end{proof}

\begin{rem}\label{rem:need-monodromy}
One might naturally ask whether Theorem~\ref{T:shadowcheb} holds without the assumption that $G_K$ acts with large monodromy on the homology of $X$.  The answer is negative; some hypothesis is needed.  Consider, for instance, a plane quartic $X/K$ of the form $Q(x^2, y^2, z^2)$ with $Q$ a nondegenerate quadratic form.  The curve $X$ naturally carries an action of $G = (\Z/2\Z)^2$ and $X/G$ has genus $0$.  Since the automorphisms are defined over $K$, the Frobenius at a prime $\mathfrak{p}$ of $\OO_K$ commutes with $G$, so the operator $\Gamma_{\Frob}$  in Definition~\ref{D:Grpthyshadow} lies in $(\wedge^2 H)^G$.  Since the Ceresa class is canonical, it too is preserved by $G$ and lies in $H^1(G_K, \wedge^3 H)^G = H^1(G_K, (\wedge^3 H)^G)$.  The shadow of Frobenius must thus also be fixed by $G$; but $H^1(G_K,H)^G = H^1(G_K,H^G)$ is trivial because $H^G$ is trivial.  We conclude that every Frobenius shadow of $X$ is zero; but we know from the example of the Fermat curve $x^4 + y^4 + z^4 = 0$ that the Ceresa class of $X$ does not need to be torsion (\cite{Bloch1}).  Indeed, in any situation where the curve has a finite group $G$ of $K$-rational isomorphisms and $X/G$ is of genus $0$, all shadows will vanish by the same argument.  

We can describe the phenomenon in the above paragraph in purely group-theoretic terms.  Write $\Pi$ for the quotient of the geometric \textup{\'{e}t}ale fundamental group of $C$ by the third term of its lower central series; it is not hard to see that the shadow of an element $\sigma \in G_K$ depends only on its image in $\Out(\Pi)$.
    
Call a subgroup $\Gamma$ of $\Out(\Pi)$ {\defi {shadowless}} if there is some nonzero integer $N$ such that the shadow of every $\gamma \in \Gamma$ is killed by multiplication by $N$.  If the restriction of the Morita cocycle to $\Gamma$ is already torsion, the subgroup $\Gamma$ is automatically shadowless; what we show in the proof of Theorem~\ref{T:shadowcheb} is that with the added assumption of big monodromy in $\GSp(H)$, the converse also holds.  But $\Out(\Pi)^G$ is a shadowless subgroup of $\Out(\Pi)$ which does not have trivial Morita cocycle.  It is an interesting group-theoretic question to understand what the maximal shadowless subgroups of $\Out(\Pi)$ are. 
    
This provides a group-theoretic account of the fact that some curves, like the Fermat quartic, have a non-torsion Ceresa class while having all Frobenius shadows trivial.  But there may be other shadows to try, arising from extra correspondences on various reductions of $X$, even when the Frobenius shadow does not suffice.  In \S~\ref{S:Bloch}, we will talk about some examples where this approach successfully certifies non-torsionness of $\nu(X,b)$ where Algorithm~\ref{alg:upperbound} does not.

\begin{question}
Is it possible to have a curve $X$ over a number field $K$ and an integer $M$ such that, for every prime $\mathfrak{p}$ of good reduction and every correspondence $Z$ on $X_{\FF_\mathfrak{p}}$, the shadow $\Sh(Z,p) \in \Pic^0(X)(\FF_\mathfrak{p})$ is killed by $M$, but $\nu(X,b)$ is not torsion?
\end{question}

\end{rem}

\section{Modified diagonal cycles of semidiagonal quartics}

\begin{defin} Following \cite{Corn}, we say that a plane curve of degree $4$ is {\em semidiagonal} if it is
defined by an equation of the form $q(x^2,y^2,z^2) = 0$, where $q$ is a nonsingular quadratic form.
\end{defin}

In this section, we study the modified diagonal cycle of the family of semidiagonal quartic curves. This family  includes the Fermat quartic, and hence its very general member has infinite order modified diagonal cycle (Lemma~\ref{L:VeryGeneral}).  We show that although the Frobenius shadow of every member in this family has order bounded by $4$ (Corollary~\ref{R:trivFrobSD}), there is an alternate shadow construction, which we denote the ``Bloch shadow", which can sometimes be used to certify nontriviality of the modified diagonal cycle (\S~\ref{S:Bloch}).

\subsection{Curves with Frobenius shadows of bounded order}

\begin{lem}\label{lem:shadow-quotient}
Let $C$ be a curve over a finite field $\FF$ with a group $G$ of automorphisms (each defined over $\FF$) such that the exponent of $\Pic^0(C/G)(\FF)$ is $e$, and let $s$ be the LCM of the orders of ($\overline{\FF}$-)point stabilizers of elements of $G$.  Let $d$ be the order of the Frobenius shadow $(2g-2)C(\FF) - s|C(\FF)|K_C]$ of $C$.  Then $d \divides es$.
\end{lem}

\begin{proof}
    Denote by $\pi$ the projection map from $C$ to $C/G$.  Enumerate the points of $C(\FF)$ as $p_1, \dots, p_n$, and group them into $G$-orbits.  Note that $|\Stab_G(p)| \sum_{q \in Gp} q$ is the pullback of the degree-$1$ divisor $\pi(p) \in \Pic(C/G)$.  It follows that $s C(\FF) := s \sum_i [p_i]$ is pulled back from $\Pic(C/G)$.  On the other hand, Riemann-Hurwitz tells us that
    $$
    K_C = \pi^* K_{C/G} + \sum_{p \in C(\overline{\FF})} (|\Stab_G(p)|-1)p.
    $$
    so $sK_C$ is also a pullback from $\Pic(C/G)(\FF)$.  So $s[(2g-2)C(\FF) - s|C(\FF)|K_C]$ is a pullback from $\Pic^0(C/G)(\FF)$, and therefore has order dividing $e$.
\end{proof}

\begin{cor}\label{cor:shadow-trivial} If $C/G$ is rational, then $d|s$. (See also Remark~\ref{rem:need-monodromy}.)
\end{cor}

\begin{cor}\label{R:trivFrobSD}
The Frobenius shadow of any semidiagonal quartic curve is $4$-torsion.
\end{cor}

\begin{proof}
Observe that all semidiagonal quartics satisfy the hypothesis of Corollary \ref{cor:shadow-trivial} with $s=4, e=1$ since $C/G \cong \mathbb{P}^1$.
\end{proof}
 In particular, Algorithm~\ref{alg:upperbound} cannot be used to certify nontriviality of the modified diagonal cycle of these curves. 
 
 \subsection{Modified diagonal cycle of a very general semidiagonal quartic via twists}
 We prove that the Fermat quartic, and hence a very general semidiagonal quartic, has nontorsion modified diagonal cycle, by showing that a certain twist of the Fermat quartic has infinite order modified diagonal cycle.

\begin{rem} The stronger statement of the modified diagonal cycle of the Fermat quartic being nontorsion in the Griffiths group (i.e., modulo algebraic equivalence) was proved by Bloch 
\cite[Theorem 4.1]{Bloch1}, extending Harris's result \cite[Section 3]{harris} of nontriviality in the Griffiths group. Their methods use additional structure special to the Fermat quartic such as having period lattice with entries in $\mathbb{Z}[i]$, and a decomposition of the Galois representation. In contrast, our method of proof, although only certifying nontriviality in the Chow group, can be applied in many situations where the approaches of Harris and Bloch cannot be easily adapted.
\end{rem}

\begin{prop}\label{P:Certwist}
If a $\mathbb{Q}$-form $X'$ of $X$ (i.e., a curve that is $\overline{\mathbb{Q}}$-isomorphic to $X$) has $\nu(X',b)$ nontorsion, then $\nu(X,b)$ is also nontorsion, whence the modified diagonal cycle $\Delta_b$ of $X$ is nontorsion in the Chow group.
\end{prop}
\begin{proof} This follows from Remark~\ref{R:CerTorGeom}.
\end{proof}

\begin{cor}\label{cor:fermat-non-torsion}
The Fermat quartic has nontorsion modified diagonal cycle in the Chow group.    
\end{cor}
\begin{proof}
We will prove that another curve $\bar{\mathbb{Q}}$-isomorphic to the Fermat quartic has nontorsion modified diagonal cycle.  
In particular, let $K$ be the cubic field $\mathbb{Q}(\zeta_7)^+$, with automorphisms $1, \sigma, \sigma^2$, and let $\ell$ the linear form 
$(-1+\alpha + \alpha^2) x + (1 + \alpha) y + (-1 - \alpha^2) z$ over $K$, where 
$\alpha = \zeta_7 + \zeta_7^{-1}$.  (The choices of $K, \ell$ are arbitrary; it appears that most choices of cubic field and linear form allow
the result to be proved.)  Then 
$\sum_{i=0}^2 (\ell^{\sigma^i})^4$ is a $\mathbb{Q}$-form of the Fermat quartic; considering the
shadows and bound for $p \in \{5,11\}$ we get $M_p = 2368$ and $N_p = 12$. It follows that its modified diagonal class is not torsion.
\end{proof}

\begin{lem}\label{L:VeryGeneral}
Let $C$ be the generic semidiagonal quartic over $\mathbb{Q}$, i.e., the curve defined by
$q(x^2,y^2,z^2) = 0$ where $q$ is the generic quadratic form in three variables over $\mathbb{Q}$
whose coefficients are independent transcendentals. Then the modified diagonal cycle of $C$ is nontorsion in the Chow group.
\end{lem}
\begin{proof}
 Let $C$ be a family of curves over an 
irreducible base.  If the class of the Ceresa cycle of the fibre of $C$ over the generic
point is annihilated by $n$ (in the Chow group, the Griffiths group, or any other 
reasonable ambient group), then the same is true for every smooth fibre of $C$.  
Combining this with Corollary \ref{cor:fermat-non-torsion}, we see that the generic
semidiagonal quartic has nontorsion modified diagonal cycle in the Chow group, and therefore that
the locus of semidiagonal quartics with torsion modified diagonal cycle is a countable union of
subvarieties of the moduli space.      
\end{proof}
\begin{rem}\label{rem:expect-not-torsion}We certainly expect that almost all semidiagonal 
quartics over a fixed number field,
when counted in any reasonable order, will have nontorsion modified diagonal cycle,
but we do not know how to prove this. Indeed, it is even plausible that there is an open dense subvariety $U$ in the space of semidiagonal quartics such that for any number field $K$ there are only finitely many quartics in $U(K)$ whose modified diagonal cycle is torsion.  When the space of curves under consideration is compact, the analogous statement follows from a Northcott property proved by Zhang~\cite[Theorem 1.3.5]{ZhangGS} for compact subvarieties of $\mathcal{M}_g$ and generalized to arbitrary subvarieties by Gao and Zhang in \cite{GaoZhangNorthcott}.
More examples and data on this question will be presented in
Example \ref{ex:use-bloch-shadow} and Section \ref{sec:census}. 
 Note also the recent results of \cite{laga-shnidmanbiellipticpicard}, which proves that in a certain family of bielliptic Picard curves $C_t$ over $K$, the modified diagonal cycle is torsion if and only if a corresponding point $P_t$ on a fixed elliptic curve $E/K$ is torsion; the finiteness of $t \in K$ such that the modified diagonal cycle on $C_t$ is torsion then follows immediately. 
\end{rem}

\begin{lem}\label{L:formSD} Let $C$ be a semidiagonal quartic over a field $K$ of characteristic not equal to $2$
such that $\Aut_{\bar K} C_{\bar K}$ is of order
$4$.  Then every form of $C$ is itself semidiagonal. 
\end{lem}
\begin{proof}
 Indeed, letting $\sigma_1, \sigma_2$ be generators
of the automorphism group, a cocycle is 
simply given by $\sigma_1 \to a, \sigma_2 \to b$ for $a, b \in \mathbb{Q}^*/(\mathbb{Q}^*)^2$.  The twist of 
$q(x^2,y^2,z^2) = 0$ by $\sigma$ is easily seen to be defined by $q(ax^2,by^2,abz^2) = 0$.
\end{proof}

\begin{rem} Lemma~\ref{L:formSD} and Remark~\ref{R:trivFrobSD} show that the argument in Corollary~\ref{cor:fermat-non-torsion} for the Fermat quartic will not prove that a typical semidiagonal quartic over a number field (i.e., one with geometric automorphism group of order $4$) has nontorsion Ceresa cycle. 
\end{rem}

\begin{rem}
There are curves defined over $\mathbb{Q}$ that become isomorphic to semidiagonal 
quartics over an extension of $\mathbb{Q}$ but are not semidiagonal over $\mathbb{Q}$ themselves: the Fricke-Macbeath quotient that is shown in
\cite[Corollary 3.8]{BLLS} to have torsion Ceresa class is of this type.  In this situation, we can use
the Frobenius shadow to certify that the Ceresa cycle is nontorsion for some examples.
This quotient has three maps to elliptic curves of $j$-invariant $1792$, which do not have CM.  It is semidiagonal over $\mathbb{Q}(\zeta_7)^+$ and has
geometric automorphism group of order $4$.  One can check from the explicit form of the twists above that it is not semidiagonal over $\mathbb{Q}$.
\end{rem}

\subsection{The Bloch shadow}\label{S:Bloch}
In this section we describe an adaptation of the methods introduced by Bloch in  \cite{Bloch1} that can be used to define shadows when the Jacobian is reducible. These generalized shadows also appear in \cite{BST}, where the authors adapt Bloch's construction to study Ceresa cycles of genus $3$ curves embedded in a triple product of an elliptic curve. An irreducibility criterion on the relevant Galois module actually allows them to show that the $\ell$-adic Abel-Jacobi map factors through the Griffiths group. 
These shadows are especially useful for the difficult case of the semidiagonal curves, for which Frobenius shadows have bounded order and thus cannot be used to certify that the Ceresa cycle has infinite order.  

We begin by recalling the construction of \cite[Section 3]{Bloch1}.  Let $C$ be the Fermat quartic, for which Bloch chooses the $\mathbb{Q}$-form $T_0^4 + T_1^4 = T_2^4$.
Then $C$ admits three independent maps to elliptic curves, two to $E: y^2 = x^3 - x$ and one to $E': y^2 = x^3 + x$.  More precisely, the
Jacobian of $C$ is $2$-power isogenous to $E \times E \times E'$ (\cite[Lemma 3.1]{Bloch1}).  
He defines a certain map $\rho: C \to E \times E \times E'$.
Thus we can define the following cycle.  Reducing mod $p$, we have a Frobenius map $\Frob: E/{\ff_p} \to E/{\ff_p}$.  Let $\Gamma_{\Frob}$ be
the graph of this map and let $B_1 = {\Gamma_{\Frob}} \times E'$, a divisor on $E \times E \times E'$.  Let $B_0 = E \times E \times 0$.  Pulling
back $B_0, B_1$ to $C$, we obtain divisors there, and these in turn may be mapped to $E'$ by the third component of $\rho$.  Bloch argues
in \cite[Section 3]{Bloch1} that the nontriviality of a certain linear combination of these divisors implies that the Ceresa class on $C$ is not algebraically equivalent to
$0$. We now define analogues of the divisors used by Bloch as follows. 

More generally, let $C$ be a curve of genus $3$ such that the Jacobian of $C$ reduces mod $p$ to a completely reducible abelian threefold isogenous to 
$E_1 \times E_2 \times E_3$, where $E_1$ and $E_2$ are themselves isogenous.  Choose maps $\rho_i: C \to E_i$ 
and let $\rho$ be the product map $\prod_{i=1}^3 \rho_i \colon C \to \prod_{i=1}^3 E_i$.  
We assume the nondegeneracy condition that no translate of a proper abelian subvariety of $\prod_{i=1}^3 E_i$ 
contains $\rho(C)$. Equivalently, by Poincar\'e's complete reducibility theorem, there do not exist isogenies $\sigma_i: E_i \to E_0$ for some elliptic curve $E_0$, and integers
$n_i$, not all $0$, such that $\sum_{i=1}^3 n_i (\sigma_i \circ \rho_i)$ is a constant map. \begin{lem}\label{L:rhoCbirational} Under this nondegeneracy assumption, the curve $C$ is birational to its image $\rho(C)$. In particular, if $C$ is a nice semidiagonal quartic as in Lemma~\ref{L:SDfacts}, then $C$ is birational onto its image under $\rho$.
\end{lem}
\begin{proof}
Let $A$ be the Jacobian of the normalization of $\rho(C)$. Then there is an abelian subvariety of $\prod_{i=1}^3 E_i$ containing  $\rho(C)$ that is isomorphic to $A$. If $C$ is not birational to $\rho(C)$, then $\dim(A) = p_g(\rho(C)) < p_g(C) = 3$, contradicting the nondegeneracy assumption on $\rho(C)$ above. Lemma~\ref{L:SDfacts}~\eqref{L:SDbir} finishes the proof.
\end{proof}

Let $\phi: E_1 \to E_2$ be a separable isogeny and let $\Frob_2$ be the Frobenius on $E_2$.
Define $D_1 \subset E_1 \times E_2 \times E_3$ to be $\Gamma_{\Frob_2 \circ \phi} \times E_3$.  Let $d_1 \colonequals \rho^* D_1$. Let $R_C$ be the ramification divisor of $\rho_3 \colon C \rightarrow E_3$. (Since $E_3$ has genus $1$, it follows by the Riemann-Hurwitz formula that $R_C$ is also the canonical divisor of $C$.) Let $D_0 \colonequals \pi_3^*(\rho_3(R_C))$ and let $d_0 \colonequals \rho^*(D_0)$. 

\begin{defin}\label{D:definingD1}\label{D:BlochShadow} 
The {\defi{Bloch shadow}} $S_{C,p;\rho}$ of $C \bmod p$ with respect to the choice of $\rho_i$ as above is the point on $E_3$ corresponding to the degree $0$ divisor
$(\deg d_1) {\rho_3}_* (d_0) - (\deg d_0) {\rho_3}_* (d_1)$.
\end{defin}

\begin{lem}\label{L:SDfacts} Let $C$ be a nice semidiagonal quartic plane curve with defining equation  $q(x^2,y^2,z^2)$ for some quadratic form $q$ in $3$ variables. Let $E_1,E_2,E_3$ be the Jacobians of the quotient curves defined by the equations $q(x,y^2,z^2),q(x^2,y,z^2)$ and $q(x^2,y^2,z)$ respectively. Then,
\begin{enumerate}[\upshape (a)]
\item The Jacobian of $C$ is isogenous to $E_1 \times E_2 \times E_3$. 
\item\label{L:SDJac} There are infinitely many primes $p$ of good reduction such that the reductions of $E_1$ and $E_2$ modulo $p$ are isogenous.  
\item\label{L:SDbir} No proper abelian subvariety of $E_1 \times E_2 \times E_3$ contains $C$. 
\end{enumerate}
\end{lem}
\begin{proof}
The first is immediate.  The second part follows from a result of
Charles \cite[Theorem 1.1]{charles}. For the third, suppose there is a proper abelian subvariety $J$ of $E_1 \times E_2 \times E_3$ containing $C$. Then the corresponding cotangent space $H^0(J,\Omega^1)$ is a proper subspace of $H^0(E_1 \times E_2 \times E_3,\Omega^1)$ that contains $H^0(C,\Omega^1)$. Since both $H^0(C,\Omega^1)$ and $H^0(E_1 \times E_2 \times E_3,\Omega^1)$ 
 are spanned by the subspaces $H^0(E_i,\Omega^1)$ for $i=1,2,3$ (for e.g. by decomposition into simultaneous eigenspaces for the action of the three involutions obtained by fixing two of the coordinates and negating the third), we see that such a $J$ cannot exist.
\end{proof}

\subsection{Canonicity of the Bloch shadow}
In this section, we study the dependence of the Bloch shadow on the choice of $\rho_i \colon C \rightarrow E_i$. 

\begin{prop}\label{prop:independent-of-rho}
Let $T_i$ be translations on $E_i$ for $i=1,2,3$.
\begin{enumerate}[\upshape (a)] 
\item\label{P:rho2ind} For $i=1,2$, the class of ${\rho_3}_* d_1$ is unaffected by replacing the $\rho_i$ by $T_i \circ \rho_i$.
\item For $i=1,2$, the class of ${\rho_3}_* d_0$ is unaffected by replacing the $\rho_i$ by $T_i \circ \rho_i$. 
\item\label{P:BlTranslate} The Bloch shadow is unchanged when we change $\rho_i$ to $T_i \circ \rho_i$ for all $i$.
\end{enumerate}
\end{prop}
\begin{proof} \hfill
\begin{enumerate}[\upshape (a)]
\item Consider the map $\Theta_{E_1,E_3} \colon \Pic^0 E_1 \to \Pic^0 E_3$ defined as follows. Let $\Delta_1 \in \Pic^0 E_1$, and let $T_1$ be the translation by $\Delta_1$ map on $E_1$. Let $\rho_1' \colonequals T_1 \circ \rho_1$. Let $\rho',D_1'$ be the analogues of $\rho,D_1$ where $\rho_1$ is replaced by $\rho_1'$. Then $\Theta_{E_1,E_3}(\Delta_1) \colonequals (\rho_3)_*(\rho^*D_1 - \rho'^*D_1')$.
For generic choice of $E_1,E_2,E_3$
the map $\Theta_{E_1,E_3}$ must be constant, because there are no nonconstant maps $E_1 \to E_3$. In fact it must be identically $0$, since $\Theta_{E_1,E_3}(0) = 0$. Since $\Theta_{E_1,E_3}$ varies continuously as we vary $E_1,E_3$, it follows that $\Theta_{E_1,E_3} \equiv 0$ even when $E_1$ and $E_3$ are isogenous. 
The argument for $T_2$ is the same as for $T_1$.  

\item The divisor $\rho^*(D_0)$ is unaffected by translations on $E_1$ and $E_2$, and hence so is $(\rho_3)_*(\rho^*(D_0))$. 
\item For $i=1,2$, this follows from the previous two parts. Now let $i=3$ and let $\rho_3' = T_3 \circ \rho_3$. Let $\rho',D_0',D_1',d_0',d_1'$ be the analogues of their unprimed versions when $\rho_3$ is replaced by $\rho_3'$.
From the definition, it is clear that $T_3(D_1) = D_1$, and hence $d_1' \colonequals \rho'^*(D_1) = \rho^*(D_1) = d_1$. It follows that
\[ (\rho_3')_*(d_1') = (\rho_3')_*(d_1)) = (T_3)_*(\rho_3)_*(d_1)). \]
Since $T_3$ is unramified and $\rho_3' = T_3 \circ \rho_3$, it follows that $R_{C,\rho_3} = R_{C,\rho_3'}$. Since $R_C \colonequals R_{C,\rho_3}$ is the ramification divisor for the maps $\rho_3$ and $\rho_3'$, it follows that $\rho_3^*(\rho_3(R_{C})) = \rho_3'^*(\rho_3'(R_{C}))$. 
Since $D_{0,\rho_3(R_C)} = \pi_3^*(\rho_3(R_C))$ and $ \rho_3 = \pi_3 \circ \rho$, by combining this with the previous sentence we get that 
\begin{equation}\label{E:d0} d_0 = \rho^*(D_{0,\rho_3(R_C)}) = \rho^*(\pi_3^*(\rho_3(R_C))) = (\rho_3)^*(\rho_3(R_C)) =  \rho_3'^*(\rho_3'(R_{C})) = d_0'.\end{equation} Combining all the above equalities, we get that 
\[ S_{C,p;\rho'} = (\rho_3')_*((\deg d_1')d_0'-(\deg d_0')d_1') = (\rho_3')_*((\deg d_1)d_0-(\deg d_0)d_1) =  (T_3)_*(S_{C,p;\rho}).\]
Since translations act trivially on $\Pic^0(E_3)$, the result follows. \qedhere
\end{enumerate}
\end{proof}

If $C$ is a semidiagonal
quartic, the effect of its automorphisms on the Bloch shadow is also easily understood.  Let $\sigma_i$ be the automorphism
of $C$ that changes the sign of the $i$th coordinate.

\begin{prop} Suppose $C$ is a semidiagonal quartic.  Let $\nu_i$ be a negation map on $E_i$.  Replacing $\rho_i$ by $\nu_i \circ \rho_i$ in the
definition of the Bloch shadow changes its sign (possibly up to torsion of order $2$).  Replacing the three
$\rho_i$ by $\rho_i \circ \sigma$, for any fixed $\sigma \in \{\sigma_1,\sigma_2,\sigma_3\}$, does not change
the Bloch shadow.
\end{prop}

\begin{proof}
By Proposition~\ref{prop:independent-of-rho}~\eqref{P:BlTranslate}, the Bloch shadow is unaffected if we compose $\rho_i$ with a translation. Combining this with the fact that two negations on the
same elliptic curve are related by translation, we see that it is sufficient to consider a single negation $\nu_i$.
The statement is now clear for $i = 3$.  For the other two, let us take $\nu_i$ to 
be the negation map that fixes the origin of $E_i$. 

The divisor $(\rho_3)_*(D_0) = (\rho_3)_* \rho_3^* (\rho_3)_* R_C$ is just $(\deg \rho_3) (\rho_3)_* R_C$.  Because $\sigma_2$ commutes with $\sigma_3$, it preserves the set of branch points of $\sigma_3$, which is to say that $(\sigma_2)_* R_C = R_C$.  So $(\rho_3)_*(R_C) = (\rho_3 \sigma_2)_* (R_C) = (\nu_3 \rho_3)_* (R_C)$.  We now know that $(\rho_3)_*(D_0)$ is fixed by $(\nu_3)_*$.

As for $D_1$, note first that
$(\nu_1 \circ \rho_1,\rho_2)^*(\Gamma_{\Frob \circ \phi}) = (\rho_1,\rho_2)^* (\Gamma_{\Frob \circ -\phi})$,
and similarly for $\nu_2$.  The automorphism $\sigma_2$ takes 
the old $D_1$ to the new $D_1'$ defined in terms of $\nu_1 \circ \rho_1$, so that in particular $D_1 + D_1'$ is fixed by $\sigma_2$.  This implies as above that $(\rho_3)_*(D_1 + D_1') \in \Pic(E_3)$ is fixed by $\nu_3$.

Now the sum of the shadow computed with $\rho_i$ and the shadow computed with $\nu_i \circ \rho_i$ is
$$2(\deg D_1) {\rho_3}_*(D_0) - (\deg D_0) {\rho_3}_* (D_1 + D_1').
$$
But this is a degree-$0$ divisor on $E_3$ fixed by
$\nu_3$, so it is $2$-torsion.

For the second statement, observe that $\rho_i \circ \sigma_j = \rho_i$ if $i = j$ and $\nu_i \circ \rho_i$ otherwise.  The second statement now follows from two applications of the first.
\end{proof}

\subsection{The Bloch shadow as a shadow of the modified diagonal}
Let $\Delta_b$ be the modified diagonal class of the curve $C$ with respect to a canonical base point $b$, i.e., a degree $1$ divisor on $C$ such that $4b = K_C$. For each $i$ in $\{1,2,3\}$, choose $\rho_i$ such that $\rho_i(b)$ is the corresponding identity element $O_{E_i}$ of the elliptic curve $E_i$.  Let $i_b \colon C \rightarrow J$ and let $i_b^{(3)} \colon C^3 \rightarrow J$ be the Abel-Jacobi maps on $C$ with respect to base point $b$.  
Let $\rho_J \colon J \rightarrow E_1 \times E_2 \times E_3$ be the induced homomorphism that satisfies $\rho = \rho_J \circ i_b$. The main theorem of this section is the following:
\begin{thm}\label{T:relatingSh} Let $\widetilde{\Delta_b} \in \CH^2(E_1 \times E_2 \times E_3)$ be the pushforward of $\MD(C,b)$ under the composite map $\rho_J \circ i_b^{(3)}$. Then
\begin{equation*}\label{E:MainEqCyc} 4 \deg(\rho_3) (\pi_3)_* \left( \widetilde{\Delta_b} \cdot  D_1 \right) = 6 \Sh_{C,p,;\rho}.\end{equation*}
\end{thm}

This has the following important corollary relating Bloch shadows to the modified diagonal cycle.
\begin{cor}\label{C:BlShMD}
If $\Delta_b$ is torsion, its order is a multiple of the order of $6\Sh_{C,p,;\rho}$.
\end{cor} 
\begin{proof}
 Let $I \colon \homtriv{E_1 \times E_2 \times E_3}{2}  \rightarrow \homtriv{E_1 \times E_2 \times E_3}{3}$ be the map corresponding to the intersection with the cycle $D_1$. If $\Delta_b$ is torsion, then its order is divisible by the order of its image (namely $6\Sh_{C,p,;\rho}$) under the following composite homomorphism:  
\[ \homtriv{C^3}{2} \xrightarrow{(\rho_J \circ i_b^{(3)})_*} 
\homtriv{E_1 \times E_2 \times E_3}{2}  \xrightarrow{(\pi_3)_* \circ I} \Pic^0(E_3)  \xrightarrow{\times 4 \deg(\rho_3)} \Pic^0(E_3).\qedhere\]
\end{proof}

Before proving Theorem~\ref{T:relatingSh}, we first prove some lemmas.  For any integer $m$ and any abelian variety $A$, let $[m]$ denote the multiplication by $m$ map on $A$ (we suppress $A$ from the notation), and let $A[m]$ denote the kernel of $[m]$.
\begin{lem}\label{L:MDPush} $\widetilde{\Delta_b} = ([3]_* -3 [2]_*+3) \rho(C).$ 
\end{lem}
\begin{proof}
A direct computation shows that $(i_b^{(3)})_*(\Delta_b) = ([3]_* -3 [2]_*+3) (i_b(C))$ (see for e.g. \cite[Proposition~5.3]{GrossSchoen}).  Since $\rho = \rho_J \circ i_b$, and $\rho(C)$ is birational to $C$ by Lemma~\ref{L:rhoCbirational}, it follows that $(\rho_J)_*(i_b(C)) = \rho(C)$. Finally since $\rho_J \circ [m] = [m] \circ \rho_J$ for every integer $m$, the equality follows.
\end{proof}

For any point $a$ on an abelian variety, let $T_a$ denote the translation by $a$ map.  For $i$ in $\{0,1\}$, let $D_i$ be as in Definition~\ref{D:BlochShadow}. 
\begin{lem}\label{L:prelim} Let $m$ be any integer.Then
\begin{enumerate}[\upshape (a)]
\item\label{L:D1pullm}
$[m]^*D_1 = \sum_{a \in E_2[m]} T_a^*(D_1).$
\item\label{L:mD1Push} $(\rho_3)_* (\rho^*([m]^* D_1)) = m^2 (\rho_3)_* (\rho^*(D_1)).$
\item\label{L:Dimult} 
$(\pi_3)_* \left( [m]_* \rho(C) \cdot D_i \right) = [m]_* (\rho_3)_* (\rho^*([m]^* D_i))$.
\item\label{L:Afterpi3D1} $(\pi_3)_* \left( [m]_* \rho(C) \cdot D_1 \right) = 
m^2[m]_* (\rho_3)_*(\rho^*(D_1))$.
\end{enumerate}
\end{lem}
\begin{proof}\hfill
\begin{enumerate}[\upshape (a)]
\item 
\begin{equation*}\label{E:D1pullm} [m]^*D_1 = [m]^*(\Gamma_{\Frob_2 \circ \phi} \times E_3) = \sum_{a \in E_2[m]} T_a^*(\Gamma_{\Frob_2 \circ \phi} \times E_3) = \sum_{a \in E_2[m]} T_a^*(D_1). \end{equation*}

\item 
\begin{align*}
(\rho_3)_* \left( \rho^* [m]^*(D_1) \right)  
&= (\rho_3)_*  \rho^* \left(
 \sum_{a \in E_2[m]} T_a^*(D_1) \right) \ \ \ \textup{(by \eqref{L:D1pullm})} \\
&=  \sum_{a \in E_2[m]} (\rho_3)_*  \rho^* T_a^*(D_1) \\
&=  \sum_{a \in E_2[m]} (\rho_3)_*  \rho^* (D_1) \ \ \ \textup{(by Proposition~\ref{prop:independent-of-rho}~\eqref{P:rho2ind})} \\
&= m^2 (\rho_3)_*  \rho^* (D_1) \ \ \ \textup{(since $\deg(E_2[m]) = m^2$)}.
\end{align*}

\item By the projection formula for $[m]$ on $E_1 \times E_2 \times E_3$, we have
\begin{align*}\begin{split}
    (\pi_3)_* \left( [m]_* \rho(C) \cdot D_i \right) &= (\pi_3)_* [m]_* \left( \rho(C) \cdot [m]^* D_i \right) \\ 
    &= [m]_* (\pi_3)_* \left( \rho(C) \cdot [m]^* D_i \right) \ \ \ \textup{(since $\pi_3 \circ [m] = [m] \circ \pi_3$)}\\
    &= [m]_* (\pi_3)_* \rho_*(\rho^*([m]^* D_i)) \ \ \ \textup{(by the projection formula for $\rho$)} \\
    &= [m]_* (\rho_3)_* (\rho^*([m]^* D_i)) \ \ \ \textup{(since $\pi_3 \circ \rho = \rho_3$)}.
\end{split}\end{align*}

\item 
\[ (\pi_3)_* \left( [m]_* \rho(C) \cdot D_1 \right) \underset{\textup{(by \eqref{L:Dimult})}}{=}  [m]_* (\rho_3)_* (\rho^*([m]^* D_1)) \underset{\textup{(by \eqref{L:mD1Push})}}{=} m^2[m]_* (\rho_3)_* (\rho^*(D_1)). \qedhere \]
\end{enumerate}
\end{proof}

\begin{proof}[Proof of Theorem~\ref{T:relatingSh}] Let $d_0,d_1$ be as in Definition~\ref{D:BlochShadow}. Note that \[(\rho_3)_*d_0 = (\rho_3)_*(\rho_3)^*(\rho_3(R_C)) = \deg(\rho_3) \rho_3(R_C).\] Since $R_C = K_C = 4b$ and $\rho_3(b) = O_{E_3}$ by our choice of $b$ and $\rho_3$, it follows that $\rho_3(R_C) = 4O_{E_3}$. Combining the previous two sentences, we get that \begin{equation}\label{E:degd0} \deg(d_0) = \deg((\rho_3)_*d_0) = 4 \deg(\rho_3),\end{equation} and that
\begin{equation}\label{E:bcan}
  (\rho_3)_*(d_0) = \deg(\rho_3)\rho_3(R_C) = 4\deg(\rho_3) O_{E_3} = \deg(d_0) O_{E_3}.   
\end{equation}

\begin{align}\label{E:MDDi}
\begin{split}
(\pi_3)_* \left( \widetilde{\Delta_b} \cdot D_1 \right) 
&= (\pi_3)_* \left( ([3]_* -3 [2]_*+3) \rho(C) \cdot D_1 \right)  \ \ \ \textup{(by Lemma~\ref{L:MDPush})} \\
&= (\pi_3)_* \left( ([3]_*\rho(C)) \cdot D_1 -3 ([2]_*\rho(C)) \cdot D_1+3 \rho(C) \cdot D_1 \right) \\
&= (9[3]_*-12[2]_*+3) \left( \rho_3)_*(\rho^*(D_1) \right) \ \ \ \textup{(by Lemma~\ref{L:prelim}~\eqref{L:Afterpi3D1})} \\
&= 6\left( (\rho_3)_*d_1 - \deg(d_1) O_{E_3}  \right) \ \ \textup{(by the definitions of $[m]_*$ and $d_1$)}.
\end{split}
\end{align}
Putting \eqref{E:degd0}, \eqref{E:bcan} and \eqref{E:MDDi} together, we get that
\begin{align*}
4 \deg(\rho_3) (\pi_3)_* \left( \widetilde{\Delta_b} \cdot D_1 \right) 
&= \deg(d_0) (\pi_3)_* \left( \widetilde{\Delta_b} \cdot D_1 \right) \\
&= 6\deg(d_0) \left( (\rho_3)_*d_1 - \deg(d_1) O_{E_3}  \right) \\
&= 6  \left( \deg(d_0)(\rho_3)_* d_1 - \deg(d_1) (\rho_3)_* d_0 \right)
\\
&= 6\Sh_{C,p,;\rho}.\qedhere
\end{align*}
\end{proof}

\subsection{An algorithm for certifying nontriviality of Ceresa cycles of semidiagonal quartics}
We can now modify Algorithm~\ref{alg:upperbound} as follows to prove that a semidiagonal quartic (or any other curve with completely reducible Jacobian) has nontorsion modified diagonal
cycle. The idea is to replace the lower bounds $M_\mathfrak{p}$ in Algorithm~\ref{alg:upperbound} obtained from the order of the Frobenius shadow (which in the semidiagonal case stay bounded by $4$ as $\mathfrak{p}$ varies by Remark~\ref{R:trivFrobSD}) by the lower bound obtained by computing the order of $6\Sh_{C,\mathfrak{p},;\rho}$ instead.

\begin{algo}\label{alg:upperboundBloch} \hfill
\begin{enumerate}[\upshape (1)]
\item Choose a nonempty finite set \(\mathcal{T}\) of auxiliary good primes of $X$.
\item For each \(\mathfrak{p}\) in $T$, compute the integer $N_{\mathfrak{p}} \colonequals \det(\Frob_{\mathfrak{p}}-I)|_M$, where $M$ is the $G_K$ module in Proposition~\ref{P:DivofTorsOrd}, and the order $M_{\mathfrak{p}}$ of $6 \Sh_{X_{\mathfrak{p}},\mathfrak{p};\rho}$.

\item Let \(\widetilde{N} = \gcd_{\mathfrak{p} \in \mathcal{T}}(p^{\infty}N_{\mathfrak{p}})\), and let \(\widetilde{M} = \lcm_{p \in \mathcal{T}}(p^{-\infty}M_{\mathfrak{p}})\).
\end{enumerate}
If $\widetilde{M} \nmid \widetilde{N}$, then output that the modified diagonal cycle of $X$ has infinite order.
\end{algo}

\begin{rem}
One could even hope to apply the same idea when the curve of genus $3$ is bielliptic, finding primes for which the simple factor of
the Jacobian becomes reducible with one of the factors isogenous to the elliptic curve factor, though this would require some good luck
since unlike Lemma~\ref{L:SDfacts}~\eqref{L:SDJac} for semidiagonal quartics, one would not expect there to be infinitely many such primes.
\end{rem}

\begin{example}\label{ex:use-bloch-shadow} Consider the curve $C \subset \pp^2$ defined by $x^4 + x^2y^2 - y^4 - x^2yz + x^2z^2 - y^2z^2 - z^4$.
In addition to the obvious automorphism $(x:y:z) \to (-x:y:z)$ it also has $(x:y:z) \to (x:-z:-y)$ and so
it is isomorphic to a semidiagonal curve.  The genus-$1$ quotients $E_1, E_2, E_3$ of $C$ all have good reduction outside $\{2,3,7,13\}$
and their $j$-invariants are $38272753/4368, 55296/7, 256000/117$.  There are
many primes $p$ for which we have $a_p(E_i) = \pm a_p(E_j)$, including $45$ of the primes up to $1000$.
In particular we find that the number of $\ff_{47}$-points on the three elliptic factors are 
$44, 44, 46$, and that the two factors with $44$ points are $2$-isogenous.  Following the procedure outlined
above, we obtain a divisor $D_1$ of degree $190$ on $C_{/\ff_{47}}$, and we calculate that 
${\rho_3}_*(4D_1 - 190D_0)$
has order $23$.  Using primes of good reduction up to $20$ to calculate the upper bound and multiplying by $6$, we obtain
$10368000 = 2^{10} \cdot 3^4 \cdot 5^3$.  It follows that the modified diagonal class of $C$ is not torsion. See the file \path{census/semidiagonal-example.mag} in \cite{code}.  
\end{example}

\section{Census of modified diagonal classes of low height smooth plane quartic curves}\label{sec:census}
As a test of the effectiveness of our methods, we decided to study the collection of all smooth plane quartic curves defined by equations
all of whose coefficients belong to $\{-1,0,1\}$.  Up to the action of $\symgp_4$ by signed permutations modulo $\pm 1$,
there are $255564$ such curves, which fall into 
$254704$ isomorphism classes over $\cc$, as is shown by calculating Dixmier-Ohno invariants.  Simply by considering our shadows on these curves for $p < 100$, we
show that the Ceresa class is torsion for at most $147$ of these, representing
$89$ different $\cc$-isomorphism classes.    Of these $147$, three of them are
$\cc$-isomorphic to $y^3 = x^4 + 1$ and one to $y^3 = x^4 + x$.  These have torsion
Ceresa cycle modulo rational equivalence by the results of \cite[Proposition~5.3(ii)]{LS}. 

On the other hand, three are forms of
the Fermat quartic and therefore have nontorsion Ceresa class modulo algebraic equivalence.  We ignore these
henceforth.  It is convenient to group the remaining $140$  curves by the
automorphism group over $\mathbb{Q}$.

\subsection{Order 1}\label{sec:census-1}
Our list of curves includes $4$ whose automorphism group over $\mathbb{Q}$ is trivial, 
all different over $\cc$.
In the given equation for each of these, there is a variable in which all monomials
have degree $0$ or $3$; in other words, these are {\em Picard curves}.  Thus the
criterion of \cite{laga-shnidman-2} can be applied.  It is easy to check that in all
four cases the point on the elliptic curve constructed in \cite[Theorem C]{laga-shnidman-2} is of finite order; therefore, the Ceresa class is torsion modulo rational equivalence.

\subsection{Order 2}\label{sec:census-2}
We find $13$ curves whose automorphism group over $\mathbb{Q}$ has order $2$.  For $6$ of
these, including one pair of curves that become isomorphic over $\mathbb{Q}(i)$,
the geometric automorphism group is of order $2$.  In one case, the failure of our 
calculations so far to prove that the Ceresa class is not torsion modulo rational equivalence
results from a 
special property of the genus-$1$ quotient: to wit, it is isogenous
to a curve with torsion subgroup of order $12$.  This contributes to the order
of the reductions of this curve mod $p$ being divisible only by primes
$\le 5$ for $p \le 100$ (but this is not guaranteed;
it is just by coincidence that $E$ does not have
$84$ points mod $89$, for example).  Since the upper bound is divisible
by powers of $3$ and $5$, this results in the order of the shadow dividing the
upper bound for all $p < 100$.  However, for $p = 101$ and $p = 113$ (and indeed for
most primes) this is no longer the case and in fact the curve does not have 
torsion Ceresa cycle.  
The other $5$ all admit maps of degree $2, 3, 6$ to 
elliptic curves.  At present we are unable to decide whether the classes of the Ceresa cycles
of these curves are torsion.   We now describe an unsuccessful attempt to determine this.
\begin{defin}
Let $C$ be a genus $g$ curve and let $\phi: C \to C'$ be a dominant map of curves of degree $d$.  Define the \defi{relative canonical shadow} $D_0(C,\phi)$ in $\Pic^0(C)$ by the formula 
\[
D_0(C,\phi) \colonequals -(2g_C-2)\phi^{-1} K_{C'} - 2 \phi^{-1} \phi(K_C) + (2dg_{C'})K_C.
\]
\end{defin}

\begin{rem} 
If $g(C') = 0$, the $(2dg_{C'})K_C$ term of $D_0(C,\phi)$ vanishes, and we are left with a degree-$0$ divisor pulled back from $\mathbb{P}^1$, which thus must be $0$.  This is as it should be, since curves with trivial Ceresa class nonetheless have lots of maps to $\mathbb{P}^1$ and it would be surprising if such maps could provide obstructions to Ceresa vanishing.  If $g(C') = 1$, then $K_{C'}$ vanishes and the formula simplifies to $D_0(C,\phi) = 2dK_C - 2\phi^{-1}(\phi(K_C))$.
\end{rem}

\begin{lem}\label{lem:canonical-pushpull}
If the modified diagonal cycle $\Delta_b$ is torsion, then the order of the relative canonical shadow $D_0(C,\phi)$ divides the order of $\Delta_b$.
\end{lem}
\begin{proof} By Lemma~\ref{L:Key}, the order of $\Sh(Z,b)$ divides the order of $\Delta_b$ for any correspondence $Z$. Consider the nice correspondence $Z \in \CH^1(C \times C)$ defined by 
\[ Z \colonequals \{ (c_1,c_2) \in C \times C \ | \ \phi(c_1) = \phi(c_2) \}.\]
In the notation of Lemma~\ref{L:ComputeSh}, we have $f_Z(b) = g_Z(b) =\phi^{-1}(\phi(b))$, $\deg(f_Z) = \deg(g_Z) = \deg(\phi) = d$. It remains to compute $Z \cdot \Delta$. Observe that $Z = (\phi \times \phi)^{-1} \Delta_{C'}$, where $\Delta_{C'}$ is the diagonal divisor on $C' \times C'$. Let $i_C \colon C \rightarrow C \times C$ denote the diagonal embedding. Similarly, define $i_{C'} \colon C' \rightarrow C' \times C'$.
Now since $\phi \times \phi \circ i_C = i_{C'} \circ \phi$, it follows that
\[ i_C^{-1}(Z \cdot \Delta) = i_C^{-1}(\phi \times \phi)^{-1}(\Delta_{C'}) = \phi^{-1} (i_{C'}^{-1} (\Delta_{C'})) = -\phi^{-1}(K_{C'}),  \]
where the last equality is by the definition of $K_{C'}$.

Lemma~\ref{L:ComputeSh} then yields
\[ 
\Sh(Z,b) = -\phi^{-1} K_{C'} - 2 \phi^{-1} \phi(b) + (2dg_{C'})b.
\]
Taking $b$ to be a canonical base point, we get
\[ 
(2g_C-2)\Sh(Z,b) = -(2g_C-2)\phi^{-1} K_{C'} - 2 \phi^{-1} \phi(K_C) + (2dg_{C'})K_C. \qedhere
\]
\end{proof}

In all five cases, it turns out that the relative canonical shadow is torsion of small order.
However, this does not prove that the class of the Ceresa cycle is torsion.

\begin{rem}\label{rem:relative-canonical-not-trivial}
In this case it is easy to verify that the relative canonical shadows are torsion and
to determine their orders, since the curve of genus $3$ is isogenous to a product of
elliptic curves.  This is not a general property of curves of genus $3$ admitting maps to
elliptic curves of degrees $2, 3, 6$:  for example, the curve
$x^2y^2 - xy^3 + y^4 + x^3z + x^2yz + y^3z + 2x^2z^2 - xyz^2 + y^2z^2 + xz^3 = 0$
(this comes from a list of curves of small conductor produced by Sutherland
\cite{Sutherland})
has Jacobian with torsion subgroup of order $5$, and yet $5$ times the relative canonical
shadow for the map of degree $6$ is not principal.  Thus the relative canonical shadow proves
that the class of the Ceresa cycle up to rational equivalence is not torsion. It is possible that the
methods of \cite{laga-shnidmanbiellipticpicard} could be adapted to the family of curves of
genus $3$ with maps to elliptic curves of degrees $2, 3, 6$, but we leave this for future work.
\end{rem}

The remaining $7$ all have automorphism
group of order $6$ over $\cc$.  These fall into two types:
there are $6$ for which the automorphism group is $S_3$ and one for which it is
cyclic of order $6$.  

In the first case, we can always change coordinates over $\mathbb{Q}(\zeta_3)$ to obtain
an equation invariant under permutations of the variables and defined over $\mathbb{Q}$.
Thus if we replace the
variables by three linear forms defined and conjugate over a cubic extension of $\mathbb{Q}$,
we obtain a new curve to which we can attempt to apply our shadow.  In each case 
we made a random choice of linear form over $\mathbb{Q}(\zeta_7)^+$ with small integral
coefficients and found that our shadow showed the Ceresa class to be of infinite
order.

The second case is the one studied in \cite{laga-shnidmanbiellipticpicard}, and it turns out to be
the case of their Theorem~1.1 with $t = 3$.  The point $(2,3)$ is torsion on the curve
$y^2 = x^3 + 1$, so the Ceresa cycle in this case is of finite order up to algebraic equivalence.

\subsection{Order 4}\label{sec:census-4}
Ignoring three for which the Ceresa cycle is known to be torsion by \cite{BS} or \cite{LS},
we have $85$ curves with $4$ rational automorphisms, which lie in $45$ different
isomorphism classes over $\mathbb{Q}$.  Of these, $4$ are forms of curves that will be
dealt with in Section~\ref{sec:census-8}.  In addition, there are $12$ with 
geometric automorphism group of order $8$, falling into $6$ isomorphism classes over
$\cc$, each containing two curves.  For three of these the automorphism group is
$C_2 \oplus C_4$.  These can be given in the form $x^4 = q(y^2,z^2)$, where
$q$ is homogeneous of degree $2$.  These can be treated in the same way as
in the general case where the geometric automorphism group has order $4$.
Namely, we have three maps to elliptic curves $E_1, E_2, E_3$.  For each of 
these curves $C$, we find a $\cc$-isomorphic curve on our list and a prime
$p$ such that two of the reductions of the $E_i \bmod p$ are
isogenous and the order of the shadow does not
divide $6$ times the initial upper bound.  This would contradict 
Corollary~\ref{C:BlShMD} if the Ceresa cycle were of finite order.

On the other hand, there are three orbits for which the geometric automorphism group is $D_4$.
In each case we can find another form of the curve for which the automorphisms are
all defined over $\mathbb{Q}$ and proceed as in Section~\ref{sec:census-8}, showing in 
each case that the Ceresa class is not torsion.  

The remaining $2$ curves are defined by the equations $x^4 + x^2y^2 - y^4 \pm z^4 = 0$, so they are isomorphic over 
$\cc$ and indeed $\mathbb{Q}(i)$.  The full automorphism group is defined over $\mathbb{Q}(\zeta_8)$.   To prove that the 
Ceresa class is not torsion modulo numerical equivalence, we use the equations for twists of plane quartics given in \cite[Prop.~5.7]{LG}.
This curve represents the case $a = \pm 1$, and in the notation used there we may take $b = 18, c = 5, m = 13, q = 1$.  
We obtain a twist defined by the equation 
$$37x^4 + 520x^3y + 2782x^2y^2 + 6760xy^3 + 6253y^4 + z^4$$
for which we can verify that the lower bound does not divide the upper bound by working only with $p<40$.

\subsection{Order 6}\label{sec:census-6}
There are $24$ curves on our list with automorphism group of order $6$, no two 
isomorphic over $\cc$.  In each case
the automorphism group can be taken to be the standard action of $S_3$ and the
technique described in Section~\ref{sec:census-2} rapidly shows that the Ceresa
class is not torsion modulo rational equivalence.

\subsection{Order 8}\label{sec:census-8}
In addition to two forms of the Fermat quartic,
we have $13$ curves with automorphism group of order $8$, representing $6$ isomorphism
classes over $\cc$ (some of which are the same as classes of curves with $4$ 
automorphisms).  In every case the automorphism group can be conjugated into 
that generated by $(x:y:z) \to (-x:y:z), (x:y:z) \to (-y:x:z)$.  In order to construct
twists, we need only look at the cyclic subgroup of order $4$ generated by the
second element shown.  Let $K/\mathbb{Q}$ be a cyclic extension of degree $4$ with Galois
group generated by $\alpha$, let $C$ be one of the $13$ curves, and consider the
action on $K(C)$ given by $(\lambda x)^\alpha = -\lambda^\alpha y, (\lambda x)^\alpha = \lambda^\alpha x, \alpha(\lambda z) = \lambda^\alpha z$.  If $\mu \in K$ satisfies
$\mu^{\alpha^2} = -\mu$ then $(\mu^\alpha x + \mu y)z$ is an invariant function.

In particular, take $K = \mathbb{Q}(\zeta_5)$ and change coordinates by $(x:y:z) \to (\zeta_5-\zeta_5^4)x + (\zeta_5^2 - \zeta_5^3)y:(\zeta_5^2-\zeta_5^3)x+(\zeta_5^4-\zeta_5)y:z)$.
This produces another form of the curve, and calculating our invariants for this form
is enough to show that $11$ of the $13$ curves have nontorsion Ceresa class
modulo rational equivalence.  The 
other two are $\cc$-isomorphic to one of the $11$, so the result follows for them
as well.

\subsection{Order 24}\label{sec:census-24}
Aside from a form of the Fermat quartic, 
there remains $1$ curve with $24$ rational automorphisms.  Its Ceresa class is shown
not to be torsion up to rational equivalence as in Section~\ref{sec:census-2}.

\subsection{Summary}
We summarize the results of our census as follows.  Of the $255564$ orbits of curves, we exclude all but $287$ by applying our
shadow for $p < 30$, and $140$ more by doing so for $p < 100$. 
The remaining $147$ curves are summarized in this table:

\begin{tabular}{|r|r|r|r|r|r|}
\hline
$\#\Aut(C/\mathbb{Q})$ & $\#\Aut(C/\cc)$ & Number of curves & Torsion & Not torsion & Undecided\\
\hline
$1$&$3$&$3$&$3$&$0$&$0$\\
$1$&$9$&$1$&$1$&$0$&$0$\\
$1$&$48$&$1$&$1$&$0$&$0$\\
$2$&$2$&$6$&$0$&$1$&$5$\\
$2$&$6$&$7$&$1$&$6$&$0$\\
$4$&$4$&$67$&$0$&$67$&$0$\\
$4$&$8$&$16$&$0$&$16$&$0$\\
$4$&$16$&$2$&$0$&$2$&$0$\\
$4$&$48$&$3$&$3$&$0$&$0$\\
$6$&$6$&$24$&$0$&$24$&$0$\\
$8$&$8$&$8$&$0$&$8$&$0$\\
$8$&$16$&$4$&$0$&$4$&$0$\\
$8$&$24$&$1$&$0$&$1$&$0$\\
$8$&$96$&$2$&$0$&$2$&$0$\\
$24$&$24$&$1$&$0$&$1$&$0$\\
$24$&$96$&$1$&$0$&$1$&$0$\\
\hline
\end{tabular}

\begin{rem}\label{rem:table-summary}
All of the undecided curves have completely reducible Jacobians with maps of degree
$2, 3, 6$ to elliptic curves.  As mentioned above, we believe that it would be possible to adapt
the methods of \cite{laga-shnidmanbiellipticpicard} to this situation and determine which
of these curves actually have torsion Ceresa class.  Aside from these, we have not discovered 
any genuinely new examples in which the Ceresa class is torsion.  Rather, we have shown that all
other candidates defined by equations of 
height $1$ are either explained by \cite[Theorem C]{laga-shnidman-2} or result
from the known examples of Beauville and Lilienfeldt-Shnidman.
\end{rem}
\begin{rem}\label{R:torsionorder}
It is also natural to ask whether our calculations suggest a way of determining the 
order of the Ceresa class when it is in fact torsion.  In particular, we might
hope that the lower and upper bounds are equal.  However, this is far from being the
case.  Using primes up to $100$ to compute both the lower and the upper bounds, we find
the following results.

\begin{tabular}{|r|r|r|r|}
\hline
Type&Number&Lower bounds&Upper bounds\\
\hline
$y^3 = x^4+x$&$1$&$6$&$2\cdot 3^8$\\
$y^3 = x^4+1$&$3$&$2,6,6$&$2^9\cdot 3^8, 2^9 \cdot 3^5, 2^9 \cdot 3^5$\\
Laga-Shnidman&$4$&$12, 12, 4, 2$&$2^3 \cdot 3^7, 2^3\cdot 3^8, 2^3 \cdot 3^8, 2\cdot 3^9$\\
2-3-6 curves&6&$4,6,4,4,4$&$2^{11},2^3\cdot 3, 2^3\cdot 3^3, 2^4 \cdot 3^3, 2^4 \cdot 5^2$\\
\hline
\end{tabular}

In most of these cases, the upper bound is much larger than the lower bound.  The case in
which they are the closest is the 2-3-6 curve whose lower bound is $6$; its upper bound
is $24$.  It is defined by the equation $x^4 + x^3y + x^3z + x^2y^2 + x^2z^2 + xy^2z + xyz^2 - y^3z - y^2z^2 - yz^3 = 0$.  We note that the lower bound for the Fricke-Macbeath quotient of
\cite{BLLS} is $36$.

\end{rem}

\bibliographystyle{alpha}
\bibliography{CeresaCertificate}

\end{document}